\documentclass[11pt]{amsart}
\usepackage{xcolor}
\usepackage[normalem]{ulem}

\usepackage[margin=2.75cm]{geometry}
\usepackage{amsthm, bbm}
\usepackage{amsmath}
\usepackage{hyperref}

\numberwithin{equation}{section} 

\theoremstyle{definition}
\newtheorem{definition}{Definition}[section]
\newtheorem{theorem}[definition]{Theorem}
\newtheorem{lemma}[definition]{Lemma}
\newtheorem{corollary}[definition]{Corollary}
\newtheorem{proposition}[definition]{Proposition}
\newtheorem{remark}[definition]{Remark}
\newtheorem{example}[definition]{Example}
\newtheorem{assumption}[definition]{Assumption}

\newcommand{\DS}{\displaystyle}

\newcommand{\bH}{{\mathbf{H}}}

\newcommand{\cA}{{\mathcal A}}
\newcommand{\cB}{{\mathcal B}}
\newcommand{\cC}{{\mathcal C}}
\newcommand{\cD}{{\mathcal D}}

\newcommand{\cF}{{\mathcal F}}

\newcommand{\cK}{{\mathcal K}}
\newcommand{\cL}{{\mathcal L}}

\newcommand{\cR}{{\mathcal R}}

\newcommand{\cT}{{\mathcal T}}
\newcommand{\cU}{{\mathcal U}}

\newcommand{\cY}{{\mathcal Y}}

\newcommand{\Pt}{{\partial }}

\newcommand{\bbC}{{\mathbb C}}
\newcommand{\bbE}{{\mathbb E}}

\newcommand{\bbN}{{\mathbb N}}
\newcommand{\bbP}{{\mathbb P}}
\newcommand{\bbR}{{\mathbb R}}

\newcommand{\bbZ}{{\mathbb Z}}
\newcommand{\bbI}{{\mathbb I}}

\def\fu{{\mathfrak{u}}}

\newcommand{\ve}{{\varepsilon}}
\newcommand{\del}{{\delta}}
\newcommand{\Del}{{\Delta}}

\newcommand{\sig}{{\sigma}}
\newcommand{\al}{{\alpha}}
\newcommand{\be}{{\beta}}
\newcommand{\ka}{{\kappa}}

\newcommand{\la}{{\lambda}}

\definecolor{OliveGreen}{rgb}{0,0.6,0}

\def\DS{\displaystyle}

\def\tX{{\tilde X}}
\def\fd{{\mathfrak{d}}}

\author{Dmitry Dolgopyat and Yeor Hafouta}
\title{Rates of convergence in CLT and ASIP for sequences of expanding maps}
\begin{document}
\maketitle


\renewcommand{\theequation}{\arabic{section}.\arabic{equation}}
\pagenumbering{arabic}

\begin{abstract} 
We prove Berry-Esseen theorems and
 the almost sure invariance principle with rates for partial sums of the form $\DS S_n=\sum_{j=0}^{n-1}f_j\circ T_{j-1}\circ\cdots\circ T_1\circ T_0$ where $f_j$ are functions with uniformly bounded ``variation" and $T_j$ is a sequence of expanding maps. 
Using symbolic representations similar result follow for maps $T_j$ in a small $C^1$ neighborhood of an Axiom A map and H\"older continuous functions $f_j$. All of our results are already new for a single map $T_j=T$ and a sequence of 
 different functions $(f_j)$.
\end{abstract}


\section{Introduction}
\subsection{Non-autonomous dynamical systems}
A great discovery made in the last century is
 that deterministic systems  could
 exhibit random behavior. One of the most notable results in this direction is the fact that ergodic averages of deterministic systems could satisfy the 
Central Limit Theorem (CLT). 
Since then  statistical properties of autonomous hyperbolic dynamical systems have been studied extensively. However,  many systems appearing in nature are time dependent due to an interaction with the outside world. In the context of dynamical systems this leads to the study of dynamics formed by a composition of different transformations rather than a single one. Such systems are often called sequential/time-dependent/non-autonomous dynamical systems.

 Many powerful tools developed for studying autonomous systems are unavailable in non autonomous setting. 
In particular, the spectral approach developed by Nagaev \cite{Nag1} 
and  extended to dynamical 
systems setting by Guivarch and Hardy \cite{GH}, provides a powerful tool for obtaining asymptotics expansions
in limit theorems for dynamical systems.
It turns out that 
{\em complex Perron Frobenius Theorem} proven by Hafouta and Kifer in \cite{HK}
(building on a previous work of Rugh \cite{Rug} and Dubois \cite{Du09}) provides a convenient tool for asymptotic computations
of the characteristic functions in the setting of Markov chains and dynamical systems.
This theorem has already found multiple applications to limit theorems \cite{DH, DDH1, Du11, HK, H-JSP, Nonlin, H-YT, H-Adv}.
The goal of the present paper is to study the rate of convergence 
in the CLT (aka Berry-Esseen theorems) and the almost sure invariance principle\footnote{See \eqref{ASIP def}.} (ASIP)  for non-autonomous dynamical systems without making any assumptions
on the growth of variance of the underlying partial sum.   In a forthcoming  work
\cite{DH LLT}  local limit theorems will be considered.

A particular case of a sequential dynamical systems  are random dynamical systems.
Ergodic theory of random dynamical systems has attracted a lot of attention in the past decades, see \cite{Kifer86, LiuQian95, Cong97, Arnold98, Crauel2002} and \cite{KiferLiu}. This includes, for instance, the theory of random invariant measures, entropy theory, thermodynamic formalism, multiplicative ergodic theory and many other classical topics in ergodic theory.
Ergodic aspects  of (non-random) sequential dynamical systems were studied for the first time in \cite{BB-84, BB-86}, see also \cite{AF}, \cite[Sections 1-2]{CR} and \cite{BBH05}.
 We refer to \cite[Section 3]{CR} and \cite{Nonlin} for examples of expanding maps and to
 \cite{Bk95, AF, PR03} for some examples of sequential hyperbolic sequences (see also Sections \ref{Eg Sec}, \ref{Comp Sec} and Appendix \ref{App C}). 

Next we discuss statistical properties.
Decay of correlations for random dynamical systems were obtained in \cite{Buzzi, Baldi, AS}. 
 Large deviations were obtained in \cite{Kifer 1996}  and the 
CLT in \cite{Kifer 1998}.  
Recently limit theorems for random dynamical systems have attracted a lot of attention. 
For hyperbolic maps in a $C^{3+\ve}$ neighborhood of a $C^{3+\ve}$ Anosov map, the CLT was obtained in \cite{DFGTV3} (see also \cite{ALS, CLBR}).  Almost sure invariance principles (ASIP, see \eqref{ASIP def})  were obtained in \cite{STE, STSU,DFGTV1, DDH CIRM, DDH, H-Adv, Su1, Su2}, while Berry-Esseen theorems (see \eqref{BE def}) were obtained in \cite{HK,DDH1, H-YT,H}. 
Concerning statistical properties of non-random expanding sequential systems, the sharp 
CLT was obtained in \cite{CR} (see also \cite{ANV, NSV}).  
  For hyperbolic systems the CLT follows from \cite{Bk95} under the assumption  that  variance of the underlying partial sum $S_n$ grows  faster than $n^{2/3}$.
Berry-Esseen theorems in the sequential setup were only obtained in \cite{Nonlin} for 
some classes of maps, under the assumption that $\text{Var}(S_n)$ grows linearly fast in $n$. 
In \cite{HNTV} ASIP rates were obtained under some growth assumptions on $\text{Var}(S_n)$ (see \S \ref{SSOurRes}).

\subsection{Our results}
\label{SSOurRes}
Our first result concerns optimal CLT rates. A classical result due to Berry and Esseen \cite{Berry, Ess0, Ess} asserts that for partial sums $\DS S_n=\sum_{j=1}^n X_j$ of zero mean
of uniformly bounded iid random variables $X_j$ the following optimal uniform CLT rates hold
\begin{equation}\label{BE def}
\Del_{0,n}:=\sup_{t\in\bbR}\left|\bbP(S_n/\sig_n\leq t)-\Phi(t)\right|=O(\sig_n^{-1}),\quad \sig_n=\|S_n\|_{L^2}
\end{equation}
where $\Phi(t)$ is the standard normal distribution function. 
By now the optimal convergence rate in the CLT was obtained for wide classes of stationary Markov chains \cite{Nag1, HH} and 
other weakly dependent random processes including 
chaotic dynamical systems \cite{RE,GH,HH,GO, Jir0,Jir}, random  dynamical systems \cite{HK,DDH1, H-YT}
uniformly bounded stationary sufficiently fast $\phi$-mixing sequences \cite{Rio},
$U$-statistics \cite{CJ,GS} and locally dependent random variables \cite{BR,BC,CS}. However, in all of these processes the variance of $S_n$ is of linear order in the number of summands $n$. 
To the best of our knowledge, the only case where optimal rate was obtained without any growth rates on the variances is for additive functional of uniformly elliptic inhomogeneous Markov chains  \cite{DH}. 
In the  present paper we obtain optimal CLT rates (in various forms) for Birkhoff sums generated by a sequence of expanding maps and 
sufficiently regular functions. This was obtained in \cite{Nonlin} for H\"older continuous functions when $\text{Var}(S_n)\geq cn$ for some $c>0$. 
Here we consider more general maps and more general functions (e.g. BV). Most importantly we will not assume any kind of growth rates for    $\text{Var}(S_n)$. 

Our main results have applications to sequence of maps in a $C^1$ neighborhood  of a given $C^2$-Axiom A map (see Appendix \ref{App C}). In \cite{DFGTV3} limit theorems were studied for random dynamical systems in $C^{3+\ve}$ neighborhood of a $C^{3+\ve}$ Anosov maps and Birkhoff sums formed by random $C^{2+\ve}$ functions. 
 Compared with \cite{DFGTV3} we can perturb more general maps in a weaker
norm, consider functions which are only H\"older continuous  and treat non-random systems.

What we will actually prove is stronger than uniform rates. In Theorem \ref{BE}(i) we will show that for every $p>0$,
\begin{equation}\label{Non U BE}
\Del_{p,n}:=\sup_{t\in\bbR}(1+|t|^p)\left|\bbP(S_n/\sig_n\leq t)-\Phi(t)\right|=O(\sig_n^{-1})
\end{equation}
The fact that we can take can positive $p$ allows applications to optimal CLT rates  in $L^p$  (see Theorem \ref{BE}(ii)), to Gaussian expectation estimates (Theorem \ref{BE}(iii)) and to optimal CLT rates in the Wasserstein distances (Theorem \ref{ThWass}).

The proof of \eqref{Non U BE} consists of several ingredients.  First, using a general result from \cite{H} it is enough to verify a logarithmic growth assumption introduced in \cite{DH}. This is done in   Proposition \ref{Growth Prop}. To prove this proposition we use ideas 
from \cite{DH} together with an appropriate martingale-coboundary decomposition. This decomposition is needed in order to obtain the moment estimates in Proposition \ref{Mom prop}, which are proven using the Burkholder inequality and ideas from \cite{CR}. 
Martingale-coboundary decomposition uses a version of real Ruelle-Perron-Frobenius (RPF), which is proven in Theorem \ref{RPF 0}. 
 The second tool needed to prove Proposition \ref{Growth Prop} is an extension of Theorem \ref{RPF 0} to complex operators (Theorem \ref{CMPLX RPF}). This result generalizes the complex RPF theorem in \cite{HK,Nonlin} and is proven using simpler arguments (see Appendix~\ref{AppPert}) which result in conditions which are easier to verify.

Almost sure invariance principle (ASIP) is
a stronger version of the CLT. The ASIP states
 that there is a coupling of $(S_n)$ with a Brownian motion $W(t)$ such that
\begin{equation}\label{ASIP def}
\delta_n:=\left|S_n-W(\sig_n^2)\right|=o(\sig_n) \,\,\,\,\text{a.s.}.
\end{equation}
The ASIP is an important statistical tool and it implies the functional CLT  and the law of iterated logarithm (see \cite{PhilStaut}). 
 Previously
the ASIP  was studied for autonmous hyperbolic systems
\cite{MN1, MN2, GO1, CM15} and their random counterparts,  \cite{DFGTV1,DDH,DDH CIRM, DHS, H-Adv, STE, STSU,  Su2, Su1}. In all of these cases $\text{Var}(S_n)$ grows linearly fast in $n$. The only existing result in which weaker growth rates are allowed  is \cite{HNTV}, where the ASIP rates $o(\sig_n^{1-\varepsilon})$ were obtained under the assumption that $\sig_n^2\geq cn^{1/2+2\kappa}$, for some $\kappa>0$.
In the present paper we combine the ideas in \cite{DDH,Haf SPL} with the aforementioned martingale methods and the sequential RPF theorem developed in the proof of \eqref{Non U BE} and obtain the ASIP rates $\delta_n=O(\sig_n^{\frac12+\varepsilon})$ for every 
$\varepsilon>0$, without growth assumptions on $\text{Var}(S_n)$.

Our setup includes a reference measure $m_0$ (e.g. Lebesgue or a time $0$ Gibbs measure)
and a notion of variation of functions (e.g. variation on $[0,1]$ or H\"older 
 continuity). Then our result hold true with respect to any initial measure $\mu_0$ which is absolutely continuous with respect to the reference measure, with bounded density with finite variation. 
For autonomous systems the fact that the weak limit is preserved when changing densities is called Eagleson's theorem \cite{Eag76} (see also \cite{GO2, ZW07}). Eagleson's theorem concerns the asymptotic behavior, and to show that the rates are persevered it is reasonable to require that the density is sufficiently regular. In \cite{BulLMS} non-stationary versions of Eagleson's theorem were discussed. When applied to the setup of this paper, the results in \cite[\S 3.3.1]{BulLMS} show that optimal uniform CLT rates are preserved under a change of density with bounded variation if the variance of the underlying sum grows linearly fast in $n$ (under one of the measures). 
 Here we show that the non-uniform optimal rate holds for all the measures in the above class  without any growth assumptions on the variance. Additionally, the above ASIP rates are preserved under such changes of measure (c.f. \cite[Corollary 1.3]{GO2} in the stationary case). 
\subsection{The layout of the paper}
In Section \ref{Sec Prim} we describe the classes of maps we consider.
Section \ref{Sec2} presents our main results. Sections \ref{Eg Sec} and \ref{Comp Sec} are devoted to examples.
In the former section we present specific examples fitting into our abstract framework: expanding and piecewise
expanding maps, subshifts of finite type, etc. In the next section we show that our abstract setting includes the
setups of \cite{CR} and \cite{Nonlin}
as special cases. Section \ref{Sec Mom} discusses the moment estimates which play a key role in our analysis.
Section \ref{Sec Log} is devoted to the analysis of of the characteristic functions near the origin.
This analysis plays a key role in the proof of our main results given in Section \ref{BE sec}. In a forthcoming 
work, including \cite{DH LLT},
we combine the results of Section~\ref{Sec Log} with the estimates of the characteristic 
function away from the origin to get further asymptotic results related to limit theorems.

The examples  in Sections  \ref{Eg Sec} and the proof of the main results rely on many non-autonomous versions of  known results for autonomous dynamical systems. Additionally, we find it more reader friendly to include a presentation of our results for hyperbolic maps  in a separate section. 
For these reasons the paper has four appendixes.
Appendix~\ref{Sec app A}  is devoted to mixing properties of sequential Gibbs measures. It generalizes the fact that for a topologically mixing expanding system on a compact manifold there is a unique absolutely continuous invariant measure, which is mixing exponentially fast for H\"older or BV observables.
 Appendix \ref{ScSFT-Gibbs} concerns extension of classical results for subshift of finite types to their non-autonomous counterparts,  including the theory of sequential 
 Gibbs measures.  Appendix \ref{App C} is about  non-autonomous hyperbolic systems formed by a sequence of maps $(T_j)$ in a small $C^1$-neighborhood of a given Axiom A  map.
We show that such systems are sequentially H\"older  conjugated to the Axiom A map, which will yield that all our results hold true for these systems. 
Lastly, Appendix \ref{AppPert}
contains a short proof of the complex sequential Ruelle-Perron-Frobenius Theorem following 
the arguments of \cite{HH}. The assumptions of our Theorem \ref{PertThm} are easier to verify in specific examples than the assumptions of the corresponding results in \cite{HK,Nonlin}.


\section{Preliminaries}\label{Sec Prim}
\subsection{Setup}
\,

Our setup here will be a sequential (uniform) version of \cite{Buzzi}.
Let $(X_j,\cB_j,m_j)_{j=0}^\infty$ be a sequence of probability spaces. For a measurable function $h:X_j\to\bbC$ 
denote $\|h\|_p=\|h\|_{p,j}=\|h\|_{L^p(X_j,\cB_j,m_j)}$. We suppose that there are notions of variation $v_j:L^1(X_j,m_j)\to [0,\infty]$ satisfying

\begin{enumerate}
    \item [(V1)] $v_j(th)=|t|v_j(h)$, for every $t\in\bbC$ and $h:X_j\to\bbC$;
    
    \item [(V2)] $v_j(g+h)\leq v_j(g)+v_j(h)$ for every $g,h:X_j\to\bbC$;
    
    \item [(V3)] $\|h\|_\infty\leq \|h\|_1+C_{\mathrm{var}}v_j(h)$ \text{for some constant }$C_{\mathrm{var}}>0$ (independent of $j$);
    
    
    \item[(V4)] \hskip-1mm the functions $\textbf{1}$ taking the constant value $1$ have finite variation and 
    $\DS \sup_j v_j(\textbf{1})\!\!<\!\!\infty$;

    \item[(V5)] there is a constant $C>0$ such that for all $j$ and $f,g:X_j\to\bbC$ we have $v_j(fg)\leq C(\|f\|_\infty v_j(g)+\|g\|_\infty v_j(f))$;
    
   \item[(V6)] if $h$ is a positive function with bounded variation bounded below by a constant $c>0$ then $v_j(1/h)\leq C(c,v_j(h)+\|h\|_{1,j})$, where $C(x,y)$ is a function on $\bbR_+\times\bbR_{+}$ which is increasing in both variables $x$ and $y$;
    
   \item[(V7)] the space of bounded functions with bounded variation is dense in $C(X_j)$ (the space of continuous bounded functions on $X_j$).
\end{enumerate}
\begin{example}\label{Eg}

(1) $X_j$ are metric spaces such that $\DS \sup_j\text{diam}(X_j)<\infty$ and $v_j$ is the H\"older constant corresponding to some exponent $\al\!\!\in\!\!(0,1]$ independent of~$j$.

(2) $X_j$ are Riemannian manifolds  such that $\DS \sup_j\text{diam}(X_j)<\infty$  and $v_j(g)=\sup|Dg|$.

(3) $X_j=[0,1]$, $m_j\!\!=\!\!\text{Lebesgue}$ and  each $v_j(g)$ is the usual variation of a function $g$. 

(4) $X_j=X$ is  bounded subset of $\bbR^d$, $d>1$ and all $m_j$ coincide with the normalized Lebsegue measure on $X$.  Moreover,
$$
v_j(g)=\sup_{0<\ve\leq \ve_0}\ve^{-\al}\int_{X}\text{Osc}(g,B_\ve(x))dx
$$
for some constants $\ve_0>0$ and $\al\in(0,1]$,
where $B_\ve(x)$ is the  ball  of radius $\ve$ around a point $x$ and for a set $A$,
$\DS
\text{Osc}(g,A)=\sup_{x_1,x_2\in A\cap X}|g(x_1)-g(x_2)|.
$
\end{example}

\vskip2mm
Next, given a measurable function $h:X_j\to\bbC$, denote
$$\|h\|_{BV}=\|h\|_{BV,j}=\|h\|_{1,j}+v_j(h).$$ Then $\|\cdot\|_{BV}$ is a complete norm on the space of bounded functions $B_j$ with finite variation. Note that $\|\cdot\|_{BV,j}$ are  equivalent to the norms $\DS \|g\|_{BV,\infty,j}=\|g\|_{\infty,j}+v_j(g),$
uniformly in $j$.

Let $T_j:X_j\to X_{j+1},$ $j\geq 0$ be a sequence of measurable maps, such that
\begin{equation}\label{V8}
    \sup_j\sup_{h: v_{j+1}(h)\leq 1}v_{j}(h\circ T_j)<\infty.
    \end{equation}
We also assume that the maps are absolutely continuous,  that is, $(T_j)_*m_j\ll m_{j+1}$. Let $\cL_j$ denote the transfer operator of $T_j$ with respect to the measures $m_j$ and $m_{j+1}$. Namely, if $\ka_j=\rho_jdm_j$ for some $\rho_j\in L^1(X_j,m_j)$ then $\cL_j \rho_j:X_{j+1}\to\bbR$ is the Radon--Nikodym derivative of the  measure $(T_j)_*\ka_j$. Then $\cL_j$ is the unique linear operator satisfying the duality relation:
\begin{equation}\label{Duality}
\int (f\circ T_j)g dm_j=\int f(\cL_j g)dm_{j+1} 
\end{equation}
for all bounded measurable functions $g$ on $X_{j}$ and $f$ on $X_{j+1}$.
Denote
$$
\cL_j^n=\cL_{j+n-1}\circ\cdots\circ\cL_{j+1}\circ\cL_j.
$$
We make three assumptions which are sequential versions of the assumptions in \cite{Buzzi}.

\begin{enumerate}
    \item [(LY1)] $\DS \sup_j\|\cL_j\|_{BV}<\infty$
    
    \item[(LY2)] There are constants $\rho\in(0,1)$, $K\geq 1$ and $N\in\bbN$ such that for every $j$ and a real function $h\in B_j$ we have 
    $$
    v_{j+N}(\cL_j^N h)\leq \rho v_j(h)+K\|h\|_1.
    $$
    
    \item [(SC)] For $a>0$ let 
    $\DS
    \cC_{j,a}=\{h\in L^1(m_j): h\geq0, \;\; v_j(h)\leq am_j(h)\}.
    $
    Then for every $a$ there are $n(a)\geq 1$ and $\al(a)>0$ such that for all $j$, 
    $n\geq n(a)$ and $h\in \cC_{a,j}$ 
    $$
    \text{ess-inf }\cL_j^nh\geq \al(a)m_j(h).
    $$
\end{enumerate}
Note that
(LY) stands for ``Lasota Yorke" and (SC) stands for ``sequential covering".

In  Section \ref{Eg Sec} we will verify the above assumptions for particular examples.

\begin{remark}
By iterating (LY2) it follows that there is a constant $K_0>0$ such that for all $j$ and all $n\geq N$  and a real function $h\in B_j$ we have
\begin{equation}\label{LY Iter}
v_{j+n}(\cL_j^n h)\leq \rho^{1+n-N} v_j(h)+K_0\|h\|_1.
\end{equation}
\end{remark}

\begin{lemma}\label{Bound1 lemma}
Under (LY1) and (LY2)
we have 
$$
\sup_j\sup_n\|\cL_j^n\textbf{1}\|_{\infty}<\infty
$$
where $\textbf{1}$ denotes the function taking the constant value $1$, regardless of its domain.
\end{lemma}
\begin{proof}
 It follows from the remark and (LY1) that   
 $\DS
\sup_j\sup_nv_{j+n}(\cL_j^n\textbf{1})<\infty.
$
Next, by property (V3) we have 
$\DS
\|\cL_j^n\textbf{1}\|_{\infty}\leq\|\cL_j^n\textbf{1}\|_{1}+ C_{var}v_{j+n}(\cL_j^n\textbf{1}).
$
To complete the proof of the lemma we note that by taking $f=\textbf{1}$ in \eqref{Duality} with $\cL_j^n$ instead of $\cL_j$ and $T_j^n$ instead of $T_j$ and using the positivity of $\cL_j^n$ we have $\|\cL_j^n\textbf{1}\|_{1}=\|\textbf{1}\|_1=1$.
\end{proof}

Appendix \ref{Sec app A} contains the following result.
\begin{theorem}\label{RPF 0}
Suppose (LY1), (LY2) and (SC).

(i) There is a sequence of positive  functions $h_j$ which is uniformly bounded in $B_j$ and uniformly bounded away from $0$ with $m_j(h_j)=1$, and constants $C>0$, $\del\in(0,1)$ such that for all $j\geq 0$ we have $\cL_jh_j=h_{j+1}$ and for all $n\in\bbN$,
\begin{equation}\label{Exp Conv Main}
\left\|\cL_j^n(\cdot)-m_j(\cdot)h_{j+n}\right\|_{BV}\leq C\del^n.
\end{equation}
(ii) Let $\mu_j:=h_j dm_j$.  Then $(T_j)_*\mu_j=\mu_{j+1}$.
Moreover if $\tilde\mu_j=g_j dm_j$ is another sequence satisfying 
\begin{equation}
\label{TildeInv}
(T_j)_*\tilde\mu_j=\tilde\mu_{j+1}
\end{equation}
then 
$\DS \lim_{n\to\infty}\|h_n-g_n\|_{1}=0$.
In fact, if $g_k \in BV$ for some $k$ then 
$\DS \lim_{n\to\infty}\|h_n-g_n\|_{BV}=0$ exponentially fast.
\end{theorem}

\begin{remark}
In general,  the measures $\mu_j$ are not unique.
That is \eqref{TildeInv}
 does {\bf not} imply 
that $\tilde \mu_j=\mu_j$ for all $j$.
 In fact, for every BV density $g_0:X_0\to\bbR$
the measures $\tilde \mu_j=(T_0^j)_*(g_0 dm_0)$ also satisfy \eqref{TildeInv}
 but in general they differ from 
$\mu_j$ (even when $T_j$ does not depend on $j$). The point is that  the density of $\tilde \mu_j$ is $g_j\!\!=\!\!\cL_0^j  g_0$. This density
satisfies $\| g_n\!\!-\!\!h_n\|_{BV}\!\!=\!\!O(\del^n)$ 
by Theorem~\ref{RPF 0}(i). However
$g_n$ differs from $h_n$ in general since $h_n$ corresponds to some possibly other choice of $g_0$.
\end{remark}

\begin{remark}\label{Remark 1.2}
Let $g\in B_{j}, f\in B_{j+n}$. Then for all $j,n$ we have
$$
m_j(g\cdot (f\circ T_j^n))=m_{j+n}((\cL_j^n g )f).
$$
 Plugging \eqref{Exp Conv Main} into the RHS we get
$\DS
|m_j(g \cdot f\circ T_j^n)-m_j(g)\mu_{j+n}(f)|\leq C_1\|g\|_{BV}\|f\|_{1}\del^n
$
for some constant $C_1>0$. Hence $m_{j+n}(f\circ T_j^n)\approx\mu_{j+n}(f)$. 
 In particular if the operators $T_j$ and measures $m_j$ are defined for all $j\in\bbZ$ so that the assumpions
(LY1), (LY2) and (SC) are valid for all $j\in\bbZ$ then 
$\DS \mu_j(f)=\lim_{n\to\infty} m_{j-n} (f\circ T_{j-n}^n)$, and so in this case
$\mu_j$ is the unique equivariant family of measures such that $\DS \sup_j \left\|\frac{d\mu_j}{d m_j}\right\|< \infty.$
\end{remark}

\subsection{Changing the reference measures.}
It is important to note that once conditions (LY1), (LY2) and (SC)
hold for a given sequence of measures $m_j$, they must hold with a uniformly equivalent sequence of measures (and hence, all the limit theorems stated in the next section
 are valid for  initial measures having BV densities with respect to $m_0$). 
This is the content of the following result.
\begin{proposition}
 Let $\mu_j $ be a sequence of probability measures such that $\mu_j=h_jdm_j$, with  $h_j$ 
 uniformly bounded and bounded away from the origin and $\DS \sup_j v_j(h_j)<\infty$.

 Then (V1)-(V7) and (LY1), (LY2) and (SC) also hold  when we repalce $m_j$ by $\mu_j$.
\end{proposition}

This proposition will allow us to prove limit theorems with respect to the Lebesgue measures (i.e. the case $m_j=$Lebesgue), as well as with respect to the unique  absolutely continuous (sequentially) invariant measures $\mu_j$ or any other sequence of equivalent measures. This is one of the advantages of the setup presented in this section.
\begin{proof}
First, it is clear that there is a constant $C_1>0$ such that for every function $g:X_j\to\bbC$,
\begin{equation}\label{Equiv}
C_1\|g\|_{L^1(m_j)}\leq \|g\|_{L^1(\mu_j)}\leq \|g\|_{L^1(m_j)}.
\end{equation}
Consequently, the $BV$ norms induced from both measures are equivalent. 
Moreover, since the measures $m_j$ and $\mu_j$ are equivalent we have $\|\cdot\|_{L^\infty(m_j)}=\|\cdot\|_{L^\infty(\mu_j)}$. Therefore (V1)-(V7) remain true if we replace $m_j$ with $\mu_j$.

Next, the transfer operator $L_j$ corresponding 
to the measures $\mu_j$ and $\mu_{j+1}$ is given by 
$$
L_jg=\frac{\cL_j(gh_j)}{h_{j+1}}.
$$
By (V6),\; $\DS \sup_jv_j(1/h_j)<\infty$. So conditions (LY1) and  (LY2) also hold true with $L_j$ instead of $\cL_j$ (possibly with different constants).

 Finally, notice that the operator $\cL_j$ is positive. Therefore, if $c$ is a constant such that $h_j\geq c>0$ then for every function $h\geq 0$,
 $$
\cL_j^n(h h_j)\geq c\cL_{j}^n(h),\,\,\, m_{j+1}\text{ a.s.}
 $$
 Now, since $\|1/h_{j+1}\|_{L^\infty}\!\!\geq\!\! D$ for some constant $D<\infty$, we conclude that for every non-negative function $h$ and all $n$ we have 
 $\DS
L_j^nh\geq Dc\cL_j^n h,\text{ a.s.}
 $
 Consequently, the validity of condition (SC) for the operators $L_j$ follows from the corresponding validity for $\cL_j$, together with \eqref{Equiv}.
 \end{proof}

\section{Main results}\label{Sec2}

Let $f_j:X_j\!\to\!\bbR,\,j\!\geq\! 0$ be a sequence of measurable functions such that $\DS \sup_j\|f_j\|_{BV}\!\!<\!\!\infty$. Set 
$\DS
S_n(x) = \sum_{j=0}^{n-1}f_j (T_0^{j} x)
$
where
$\DS
T_k^j=T_{k+j-1}\circ\cdots\circ T_{k+1}\circ T_k.
$

We consider the sequence of functions $S_n:X_0\to\bbR$ as random variables on the probability space $(X_0,\cB_0, m_0)$. Our results will be limit theorems for such sequences.



Denote $\sig_n=\sqrt{\text{Var}_{m_0}(S_n)}$ and $\bar S_n=S_n-m_0(S_n)$.
Let $F_n(t)=\bbP_{m_0}(\bar S_n/\sig_n\leq t)$ be the distribution function of 
$\bar S_n/\sig_n$, and let $$\Phi(t)=\frac{1}{\sqrt {2\pi}}\int_{-\infty}^t e^{-x^2/2}dx
$$
be the standard normal distribution function.

Recall that the (self-normalized) central limit theorem (CLT) means that for every real $t$ 
$$
\lim_{n\to\infty}F_n(t)=\Phi(t).
$$
The CLT in our setup can be proven using  martingale coboundary decomposition 
of \S \ref{SSMartCob}
and applying an appropriate CLT for martingales (cf. \cite{GL, CR}).
We refer to \S \ref{SSMartCob} for a characterization when $\sig_n\to\infty$.
 In this paper we will not give a separate proof of the CLT since our main results
give not only the CLT but also rate of convergence, in various metrics.

\begin{theorem}\label{BE}
Suppose $\sig_n\to\infty$.

\vskip0.2cm
(i)  (A non-uniform Berry-Esseen theorem). $\forall s\geq 0$ there is a constant $C_s$ such that 
$$
\sup_{t\in\bbR}(1+|t|^s)\left|F_n(t)-\Phi(t)\right|\leq C_s\sig_n^{-1}.
$$ 
\vskip0.1cm
(ii) (A Berry-Esseen theorem in $L^p$). For all $p\geq1$ we have 
$\DS
\left\|F_n-\Phi\right\|_{L^p(dx)}=O(\sig_n^{-1}).
$
\vskip0.1cm
(iii)  For all $s\geq 1$  there is a constant $C_s$ such that
for every absolutely continuous function $h:\bbR\to\bbR$  such that 
$H_s(h):=\int \frac{|h'(x)|}{1+|x|^s}dx<\infty$ we have 
$\DS
\left|\bbE_{m_0}[h(\bar S_n/\sig_n)]-\int h d\Phi\right|\leq C_s H_s(h)\sig_n^{-1}.
$

\end{theorem}

Next, recall that the $p$-th Wasserstien distance between two probability measures $\mu,\nu$ on $\bbR$ with finite absolute moments of order $p$ is given by 
$$
W_p(\mu,\nu)=\inf_{(X,Y)\in\mathcal C(\mu,\nu)}\|X-Y\|_{L^p}
$$
where $\mathcal C(\mu,\nu)$ is the class of all  pairs of random variables $(X,Y)$ on $\bbR^2$ such that $X$ is distributed according to $\mu$, and $Y$ is distributed according to $\nu$. 
 Combining our estimates with
the main results in \cite{H} also yields the following result.

\begin{theorem}\label{ThWass}[A Berry-Esseen theorem in $W_p$]
For every $p$ we have 
$$
W_p(dF_n, d\Phi)=O(\sig_n^{-1})
$$
where $dG$ is the measure induced by a distribution function $G$.
\end{theorem}


A key ingredient in the  proof of Theorems \ref{BE} and \ref{ThWass} is the following proposition, which we believe has its own interest.
\begin{proposition}\label{Mom prop}
For every $p\!\!\geq\!2$ there is a constant  $C$ such that for all $j\!\!\geq \!0$ and 
$n\!\in\!\bbN$,
$$
\|\bar S_{j,n} f\|_{L^p(m_0)}\leq C(1+\|\bar S_{j,n}\|_{L^2}).
$$
Moreover, $C$ depends only on $p$,\; $\DS\sup_j\|f_j\|_{BV}$, and the constants from assumptions (V1)--(V7),
(LY1), (LY2) and (SC).
\end{proposition}

Finally, we also prove an almost sure invariance principle, with rates.

\begin{theorem}\label{ASIP} 
Denote $V_n=\text{Var}(S_n)=\sig_n^2$.
For every $\ve>0$ there is a coupling of the sequence of random variables $f_j\circ T_0^j$ (on $(X_0,\cB_0,m_0)$) with a sequence of independent centered normal random variables $Z_j$ such that:

\vskip0.2cm
(i) $\DS \left|S_n-\mu_0(S_n)-\sum_{j=1}^{n}Z_j\right|=o(V_n^{1/4+\ve})$ a.s.;

\vskip0.2cm
(ii) $\DS \left\|S_n-\mu_0(S_n)-\sum_{j=1}^{n}Z_j\right\|_{L^2}=O(V_n^{1/4+\ve})
$;
\vskip0.2cm

(iii) Var$\DS\left(\sum_{j=1}^n Z_j\right)-V_n=
    O\left(V_n^{1/2+\ve}\right).
    $
    \end{theorem}

\section{Examples}\label{Eg Sec}

 Here we exhibit several classes of systems fitting in the abstract setup of Section \ref{Sec Prim}.

\subsection{Piecewise expanding maps on the interval}
We take $X_j=I=[0,1]$ and $m_j=$Lebesgue for all $j$. Let $v_j=v$ be the usual variation of functions on $[0,1]$: 
$$
v(g)=\sup_{n}\sup_{t_0=0<t_1<...<t_{n}<t_{n+1}=1}\sum_{j=0}^n|g(t_{n+1})-g(t_n))|.
$$

We suppose that for each $j$ we can write $\DS
[0,1]=\bigcup_{k=1}^{d_j}I_{j,k}
$
where   $I_{j,k}, 1\leq k\leq d_j$ are intervals with disjoint interiors such that $\DS \sup_j d_j<\infty$.

We also suppose that for all $j$ and $k$ the restriction $T_{j,k}:=T_j|I_{j,k}$ is a $C^2$
expanding map so that 
$\DS
\sup_{j}\max_{1\leq k\leq d_j}\|T''_{j,k}\|_\infty<\infty\text{ and }\,\del:=\inf_{j}\min_{1\leq k\leq d_j}\inf|T_{j,k}'|>1.
$
Moreover, we assume that 
$$
\inf_{j}\min_{1\leq k\leq d_j}|I_{j,k}|>0.
$$
Let $\cL_j$ be the transfer operator of $T_j$, namely the operator given by
$$
\cL_j g(x)=\sum_{k:\, x\in T_{j,k}(I_{j,k})}\frac{g(T_{j,k}^{-1}x)}{T_j'(T_{j,k}^{-1}x)}.
$$
Then (LY1) holds.
 A standard argument (see \cite{LY}) yields that if $N$ satisfies $\del^N>2$ then there are constants $K\geq1$ and $\rho\in(0,1)$ such that for all $j$ we have
$\DS
\text{var}(\cL_j g)\leq \rho\;
\text{var}(g)+\|g\|_{L^1}.
$
Thus, (LY2) holds. 

Next, in order to verify (SC) we can assume that for every interval $J\subset [0,1]$ there is $n(J)\in\bbN$ such that for every $j$ we have 
\begin{equation}\label{Covv}
T_j^{n(J)} J=[0,1].
\end{equation}
Under the above condition the verification  (SC) is carried out similarly to \cite[\S 1.2]{Buzzi}. 

 For piecewise expanding maps satisfying the assumptions described above
we get all the limit theorems described in the previous section for sums of the form 
$\DS
S_n=\sum_{j=0}^{n-1}f_j\circ T_0^j
$
considered as random variables with respect to a measures which is absolutely continuous with respect to Lebesgue and its density is a BV function. Here, $f_j$ must satisfy $\DS \sup_j\|f_j\|_{BV}<\infty$.

\subsection{High dimensional piecewise expanding maps.}\label{high dim}

Let $X_j=X$ coincide with a single compact subset of $\bbR^k$ for some $k>1$. Let $m$ be the normalized Lebesgue measure on $X$ and let $v$ be the variation defined in Example \ref{Eg}(iv). 

We suppose that the maps $T_j$ have the following properties. There are constants 
$d\in\bbN$, $\gamma,C,\ve>0$ and $s\in(0,1)$ with the following properties.
For each $j$ there are disjoint sets $A_{j,i}, \tilde{A}_{j,i}, 1\leq i\leq d_j\leq d$
and maps $T_{j,i}:\tilde A_{j,i}\to X$ such that:

\vskip0.2cm
(i) The sets $A_{j,1},...,A_{j,d_j}$ are disjoint and 
$\DS m(X\setminus\bigcup_{i=1}^{d_j}A_{j,i})=0$. 
Moreover, $A_{j,i}\subset\tilde A_{j,i}$;
\vskip0.2cm

(ii) Each $T_{j,i}$ is a $C^{1+\gamma}$ function;
\vskip0.2cm

(iii) $T_j|_{A_{j,i}}=T_{j,i}$  and 
for all $j$ and $i$ and $B_\ve(T_{j,i}A_{j,i})\subset T_{j,i}(\tilde A_{j,i})$, where $B_\ve (A)$ is the $\ve$-neighborhood of a set $A$;
\vskip0.2cm
(iv) For all $j$ and $i$ the function $J_{j,i}=\text{Det}(D T_{j,i}^{-1})$ satisfies that for all $x,y\in T_{j,i}(A_{j,i})$,
$$
\left|\frac{J_{j,i}(y)}{J_{j,i}(x)}-1\right|\leq C\text{dist}(x,y)^{\gamma};
$$
\vskip0.2cm

(v) For every $x,y\in T_{j,i}(\tilde A_{j,i})$ with $\text{dist}(x,y)\leq \ve$ we have 
$\DS
\text{dist}(T_{j,i}^{-1}x,T_{j,i}^{-1}y)\leq s\cdot \text{dist}(x,y);
$
\vskip0.2cm
(vi) Each $\partial A_{j,i}$ is co-dimension one embedded compact $C^1$-submanifold and 
$$
s^\gamma+\frac{4s}{1-s}Z\frac{\Gamma_k}{\Gamma_{k-1}}<1
$$
where $\Gamma_k$ is the  volume of the unit ball in $\bbR^k$ and 
$\DS Z=\sup_j\sup_x\sum_{i}\bbI(x\in \tilde A_{i,j})$. 
\vskip0.2cm

\noindent
Under the above assumptions (LY1) and (LY2) with $N\!\!=\!\!1$ are satisfied (see \cite[Lemma 4.1]{Sau}).

Next, we also assume that for any open set $U$ there exists $n(U)\in\bbN$ such that for all $j$ 
\begin{equation}
\label{UCover}
T_j^{n(U)}U=X.
\end{equation}
Under the above condition the verification (SC) is carried out similarly to \cite[Lemma 3]{DFGTV1}.

One example where \eqref{UCover} holds are {\em Markov} maps. That is, we assume that 
for each $i, j$ the image $T_j A_{ij}$ is a union of some of the sets $A_{j+1, k}$. Moreover suppose that the system is
uniformly mixing in the sense that
$\exists \ell$ such that  $T_j^\ell A_{ij}=X_{j+\ell}=X$ for each $i, j$. This condition can be verified as follows.
Consider the adjacency matrix $\cA_{j}$ such that $\cA_j(i, k)=1$ if $T_j A_{j, i}\supset A_{j+1, k}$.
Then the uniform mixing assumption means that for each $j$ all entries of
$\cA_{j+\ell-1} \cdots\cA_{j+1} \cA_{j}$ are positive. Now let $U$ be an open set in $X_j=X.$ Without loss of generality we
may assume that $U$ is an open ball with center $x$ and radius $r.$
Given $k$ let $\cB_{j,k}(x)$ denote the set of points $y\in X$ such that 
$T_j^m x$ and $T_j^m y$ belong to the same elements of our partition $(A_{j+m,q})_q$ for all $m<k.$
By our assumptions $\text{diam}(\cB_{j,k}(x))\leq C_0 s^k$ for some constant $C_0>0$, and so for sufficiently 
large $\bar k$ we have
$\cB_{j,\bar k}(x)\subset U.$ Accordingly, $T_j^{\bar k} U$ contains one of elements of our Markov partition and then 
$T_j^{\bar k+\ell} U=X.$

\subsection{Covering maps and sequential SFT}\label{Eg Gibbs}
 Suppose that each $X_j$ is a metric space. Let $\mathsf{d}_j$ be the metric on $X_j$ and suppose that $\text{diam}(X_j)\leq 1$. 
Let $v_j=v_{j,\al}$ be the H\"older constant corresponding to some fixed exponent $\al\in(0,1]$.

 \begin{assumption}\label{AssPairing} {\bf (Pairing).}
There are constants $\xi\leq 1$ and $\gamma>1$ such that for every two points $x,x'\in X_{j+1}$ with $\mathsf{d}_{j+1}(x,x')\leq \xi$ we can write 
 $$
T_j^{-1}\{x\}=\{y_i(x): i\leq k\}, \quad T_j^{-1}(x')=\{y_i(x'): i\leq k\}
 $$
where
$\DS 
\mathsf{d}_j(y_i(x),y_i(x'))\leq \gamma^{-1}\mathsf{d}_{j+1}(x,x')
$
for all $i$.

\noindent
 Moreover $ \DS \sup_j \deg (T_j)<\infty$, where   $\deg(T)$ is the largest number of preimgaes that a point $x$ can have under the map $T$.
\end{assumption}
Denote by $B_j(x,r)$ the open ball of radius $r$  around a point $x\in X_j$.

 \begin{assumption}\label{Ass n 0}
 {\bf (Covering).}
 There exists $n_0\in\bbN$ such that
for every $j$ and $x\in X_j$ we have 
\begin{equation}\label{CoverLLT}
T_j^{n_0}\left(B_j(x,\xi)\right)=X_{j+n_0}.
\end{equation}

\end{assumption}
 Fix some $\al\in(0,1]$ and a  sequence of functions $\phi_j:X_j\to\bbR$ such that 
$\DS \sup_j\|\phi_j\|_\al<\infty$. Here $\|\phi_j\|_\al=\sup|\phi_j|+v_j(\phi_j)$  and $v_j(\phi_j)$ is the 
H\"older constant of $\phi_j$ corresponding to the exponent $\al$. 
Let $L_j$ be the operator which maps a function $g:X_j\to\bbR$ to a function $L_j g:X_{j+1}\to\bbR$ given by $L_j g(x)=\sum_{T_j y=x}e^{\phi_j(y)}g(y)$. Then (see \S \ref{Gibbs}) there is a sequence of probability measures $\nu_j$ on $X_j$ such that $(\cL_j)^*\nu_{j+1}=\la_j\nu_j$, where $\la_j>0$ is bounded and bounded away from $0$. Then we can take any measure $m_j$ of the form $m_j=u_jd\nu_j$ with $\sup_j\|u_j\|_\al<\infty$. This includes the unique sequence of measures $\mu_j$ which are absolutely continuous with respect to $\nu_j$ and $(T_j)_*\mu_j=\mu_{j+1}$ (see \S \ref{Gibbs}). This setup includes the following more concrete examples.

\subsubsection{Smooth expanding maps}
Let $M$ be $C^2$ compact connected  Riemannian manifold, and let $X_j=M$ for all $j$. 
Let $T_j:M\to M$ be $C^2$ endomorphisms of $M$ such that 
$$
\sup_j\|DT_j\|<\infty\text{ and }\sup_j\|(DT_j)^{-1}\|<1.
$$
Then the arguments in \cite[Section 4]{Kifer Thermo} (see \cite[(4.6)]{Kifer Thermo}) yield that Assumption \ref{AssPairing} is in force with some $\xi>0$. Next, arguing like at the paragraph below \cite[(4.19)]{Kifer Thermo}) we see that Assumption~\ref{Ass n 0} holds if $n_0$ is large enough. 
Take $\phi_j=-\ln J_{T_j}$. Then by \cite[Theorem 3.3 and Proposition~3.4]{Nonlin} (see also \cite[Theorem 2.2]{Kifer Thermo}) we see that the measures $\nu_j$ described after Assumptions \ref{Ass n 0} coincide with the normalized volume measure on $M$. Thus, we get all the limit theorems  with respect to any absolutely continuous measure whose density is H\"older continuous.

\subsubsection{Subshifts of finite type}
Let $\mathcal A_j=\{0,1,...,d_j-1\}$ with $\DS \sup_j d_j<\infty$. Le
 $A^{(j)}$ be matrices of sizes $d_j\times d_{j+1}$ with 0-1 entries. We suppose that there exists an $M\in\bbN$ such that for every $j$ the matrix $A^{(j)}\cdot A^{(j+1)}\cdots A^{(j+M)}$ has positive entries. 
Define 
\begin{equation}
\label{DefOneSided}
X_j=\left\{(x_{j,k})_{k=0}^\infty: \,x_{j,k}\in \cA_{j+k}, A^{(j+k)}_{x_{j,k}, x_{j,k+1}}=1\right\}.
\end{equation}
Let $T_j:X_j\to X_{j+1}$ be the left shift.  
Consider a metric $\mathsf{d}_j$ on $X_j$ given by 
$$
\mathsf{d}_j(x,y)=2^{-\inf\{k:\, x_{j,k}\not= y_{j,k}\}}.
$$
With this metric  the maps $T_j$ satisfy Assumptions \ref{AssPairing} and \ref{Ass n 0}. In order to introduce appropriate measures $m_0$ note that we can extend  the dynamics for negative times by
 defining $X_j$ for $j\!\!<\!\!0$ via appropriate extensions of the sequences $(d_j)$ and $(A^{(j)})$.  Note that such extensions are highly non-unique.
 Each one of these extensions gives raise to a unique time $0$ Gibbs measure $\mu_0$ (see Appendix \ref{ScSFT-Gibbs}).
Thus, all of our results hold true when starting with any measure which is absolutely continuous with H\"older continuous  density  with respect to  $m_0=\mu_0$.


\subsubsection{Two sided SFT}\label{2side}
Using the same notations like in the previous section but also considering negative integers $j$, we define 
\begin{equation}
\label{DefTwoSided}
\tilde X_j=\left\{(x_{j,k})_{k=-\infty}^\infty: \,x_{j,k}\in \cA_{j+k}, \;\; A^{(j+k)}_{x_{j,k}, x_{j,k+1}}=1\right\}.
\end{equation}
Let $\tilde T_j:\tilde X_j\to \tilde X_{j+1}$ be the left shift, and let 
$\DS
\tilde T_j^n=\tilde T_{j+n-1}\circ\cdots\circ\tilde T_{j+1}\circ\tilde T_j.
$
Consider the metric $\mathsf{\tilde d}_j$ on $\tilde X_j$ given by 
$\DS
\mathsf{\tilde d}_j(x,y)=2^{-\inf\{|k|:\, x_{j,k}\not= y_{j,k}\}}.
$
Let 
 $\pi_j:\tX_j\to X_j$ be
the natural projection.  

By Lemma \ref{Sinai} (proven in Appendix \ref{ScSFT-Gibbs}), given a sequence of functions $\textbf{f}_j:\tilde X_j\to\bbR$  with $\DS \sup_j\|\textbf{f}_j\|_\al<\infty$ 
 there is sequences of functions $f_j:X_j\to\bbR$ and $u_j:\tilde X_j\to\bbR$ such that 
 $\DS \sup_j\|f_j\|_{\al/2}<\infty$,  $\DS\fu= \sup_j\|u_j\|_{\al/2}<\infty$ and 
 $$
\textbf{f}_j=f_j\circ\pi_j+u_{j+1}\circ \tilde T_j-u_j.
 $$
Therefore, denoting $\DS S_n\textbf{f}=\sum_{j=0}^{n-1}\textbf{f}_j\circ \tilde T_0^j$ we have
 $\DS
 \sup_n\sup|S_n\textbf{f}-S_nf|<\infty.
 $

Next, let 
$\gamma_j$ denote  a Gibbs measure of the two sided shift at time $j$ 
(see Appendix~\ref{ScSFT-Gibbs}). 
 Then $\mu_j=(\pi_j)_*\gamma_j$
are also Gibbs measures for one sided subshift and they satisfy the assumptions
of Section \ref{Sec Prim} (see Appendix \ref{ScSFT-Gibbs} for details).
We view $S_n\textbf{f}$ as random variables on $(\tilde X_0,\text{Borel},\gamma_0)$.
 Then 
\begin{equation}\label{App}
A:=\sup_n\|S_n\textbf{f}-S_nf\|_{L^\infty(\gamma_0)}<\infty.
\end{equation}

This is  enough to do deduce all of our results for $S_n\textbf{f}$, relying on the corresponding results for $S_n f$. Indeed, part (ii) of Theorem \ref{BE} is a direct consequence of part (i).  Theorem \ref{BE}(iii) 
also follows from Theorem \ref{BE}(i). Indeed, 
for every random variable $W$ with distribution function $F$ and a function $h$ 
satisfying $H_s(h)<\infty$
we have
$$
E[h(W)]-h(\infty)
=-E\left[\int_{W}^{\infty}h'(x)dx\right]=-\int_{-\infty}^{\infty}h'(x)P(W\leq x)dx= -\int_{-\infty}^{\infty}h'(x)F(x)dx.
$$
To show that Theorem \ref{BE}(i) for $S_nf$ implies Theorem \ref{BE}(i) for  $S_n\textbf{f}$,  
let 
$$
F_n(t)\!=\!\bbP\left(\frac{S_n\textbf{f}}{\sig_n}\!\leq\! t\right),
\quad G_n(t)=\bbP\left(\frac{S_n f}{\kappa_n}\leq t\right)
\text{ where }\kappa_n=\|S_nf\|_{L^2}
\text{ and }\sig_n=\|S_n\textbf{f}\|_{L^2}.$$ 
By \eqref{App}  and the triangle inequality, $|\sig_n-\kappa_n|\leq A$. To complete the proof
 fix $s\geq 0$. Then
$$
F_n(t)\leq G_n(t\sig_n/\kappa_n+A/\kappa_n)\leq \Phi(t\sig_n/\kappa_n+A/\kappa_n)+C_s\left(1+|t\sig_n/\kappa_n+A/\kappa_n|^s\right)^{-1}
$$
$$
\leq \Phi(t)+C(1+|t|)\sig_n^{-1}e^{-ct^2}+\tilde C_s(1+|t|^s)^{-1}
$$
where in the penultimate inequality we have used that $|\Phi(x+\ve)-\Phi(x)|\leq C\ve e^{-x^2/2}$ for every $x$ and $\ve>0$ and that $|t\sig_n/\kappa_n+A/\kappa_n-t|\leq C(|t|+1)/\sig_n$ for some constant $C>0$.
Similarly, 
$$
F_n(t)\geq G_n(t\sig_n/\kappa_n-A/\kappa_n))\geq \Phi(t)-C(1+|t|)\sig_n^{-1}e^{-ct^2}-\tilde C_s(1+|t|^s)^{-1}.
$$

Finally to deduce Theorem \ref{ThWass} for $S_n\textbf{f}$ from the corresponding result for $S_nf$  let us fix some $p\geq1$. Then by Theorem \ref{ThWass}, we can couple $S_nf$ with a standard normal random variable $Z$ 
so that $\|S_n f/\sig_n-Z\|_{L^p}\leq C\sig_n^{-1}$. Now by Berkes--Philipp Lemma \cite[Lemma A.1]{BP},  we can also couple all three random variables $S_n f$, $S_n\textbf{f}$ and $Z$ so that \eqref{App} still holds under the new probability law.

\section{Verification of (LY1), (LY2) and (SC) for some classes of maps}\label{Comp Sec}
In this section we will show that the conditions in Section \ref{Sec2} are in force for the classes of expanding maps considered in both \cite{CR} and \cite{Nonlin}.

\subsection{Verification of our assumptions in the setup of \cite{CR}}\label{Comp1}
 The following assumption is taken form
\cite{CR}.

\begin{assumption}\label{CRass}
\vskip0.2cm
(i) $(X_j,\cB_j,m_j)$ 
coincide with the same probability space $(X,\cB,m)$, $v_j=v$ does not depend on $j$.
\vskip0.2cm
 (ii) Conditions (LY1) and (LY2) hold with $m_j=m$ and $v_j=v$.
\vskip0.2cm
 (iii)  There is a constant $\delta_0>0$ such that
\begin{equation}\label{Min}
\text{ess-inf } \cL_0^n \textbf{1}\geq \del_0
\end{equation}
 (note that the condition \eqref{Min} is denoted by (Min) in \cite{CR}). 
Here $\textbf{1}$ denotes the function taking the constant value $1$ on $X$.
\vskip0.2cm
 (iv)   There are constants $C_1>0$ and $\del_1\in(0,1)$ such that for every $h\in B$ with $m(h)=0$ 
\begin{equation}\label{CR RPF}
\|\cL_j^n\|_{BV}\leq C_1\del_1^n.
\end{equation}

\end{assumption}

\begin{lemma}
Assumption \ref{CRass} implies the covering condition (SC). 
\end{lemma}

Thus, Assumption \ref{CRass} is less general than the combination of (LY1), (LY2) and (SC).

\begin{proof}
Since (LY1) and (LY2) hold, Lemma \ref{Bound1 lemma} implies that 
$\DS C\!\!=\!\!\sup_k\|\cL_0^k1\|_\infty\!\!\in\!\!(0,\infty)$. Hence using the positivity of the operators $\cL_k$ we see that $m$-a.s. we have
$$
\del_0\leq \cL_{0}^{j+n}\textbf{1}=\cL_j^n(\cL_0^j\textbf{1})\leq C\cL_j^n\textbf{1}.
$$
We thus conclude  that 
$\DS
\inf_j\text{ess-inf }\cL_{j}^n\textbf{1}\geq \del_2=\del_0/C.
$

Next, fix some $a$ and let $h\in \cC_{a,j}$  where $\cC_{a,j}$ comes from the condition (SC). 
Denote $\|\cdot\|=\|\cdot\|_{BV}$. Then by \eqref{CR RPF} we have
$$
\|\cL_j^n  h-m(h)\cL_j^n\textbf{1}\|=\|\cL_j^n(h-m(h)\textbf{1})\|\leq C_1\|h-m(h)\textbf{1}\|\del_1^n\leq C_2\|h\|\del_1^n\leq C_3(a)m(h)\del_1^n
$$
for some constant $C_3(a)$ which depends only on $a$ (we can take $C_3(a)=C_2(1+a)$ since $\|h\|=v(h)+m(h)\leq (1+a)m(h)$). Using that $\cL_j^n1\geq \del_2>0$ we conclude that 
$$
\cL_j^n h\geq \del_2m_j(h)-C_3(a)m(h)\del_1^n=(\del_2-C_3(a)\del_1^n)m(h)\quad
\text{($m$-a.s.)}.
$$
Let $n(a)$ be the smallest positive integer such that $C_3(a)\del_1^{n(a)}\leq \frac12 \del_2$. Then for $n\geq n(a)$ we have 
$\DS
\cL_j^n h\geq \frac12\del_2 m(h)
$ ($m$-a.s.)
and so (SC)  holds with $\al(a)=\frac12\del_2$. 
\end{proof}

\subsection{Verification of our assumptions   in the setup of \cite{Nonlin}}\label{Gibbs}
 Consider maps $T_j$  from \S \ref{Eg Gibbs}. We also suppose that we can extend the sequence $(T_j)_{j\geq0}$ to a two sided sequence $(T_j)_{j\in\bbZ}$ with the same properties. This is possible if there is a map $T_{-1}:X_0\to X_0$ such that the sequence $(T_{-1}^{j})_{j\geq0}$ has same paring property and covering assumption like the sequence $(T_j)_{j\geq0}$ (this is the case when $X_0=X_1$). Indeed, in this case we can define $T_j=T_{-1}$ for $j<0$. The need in a two sided sequences in this context arises from the lack of a given reference measure $m_0$, as will be elaborated in what follows.


Fix some H\"older exponent $\al\in(0,1]$ and let $\phi_j:X_j\to\bbR,$ $j\in\bbZ$ be such that $\DS \sup_j\|\phi_j\|_\al<\infty$.
 Let the operator $\textbf{L}_j$ be given by
$$
\textbf{L}_jg(x)=\sum_{y: T_jy=x}e^{\phi_j(y)}g(y).
$$
Then as proven\footnote{This requires us to have a two sides sequence.} in \cite{Nonlin}, there are strictly positive functions $\bar h_j\in B_j$, probability measures $\nu_j$ on $X_j$ and numbers $\la_j>0$ such that $\DS 0<\inf_j\inf \bar h_j\leq \sup_j\|\bar h_j\|_{\al}<\infty$ and 
\begin{equation}\label{ExpConv}
\left\|\textbf{L}_{j}^n/\la_{j,n}-\nu_j\otimes \bar h_{j+n}\right\|_{\al}\leq C\del^n, \quad \del\in(0,1)
\end{equation}
where $\DS \la_{j,n}=\prod_{k=j}^{j+n-1}\la_k$ and $(\nu\otimes h)$ is the linear operator $g\to \nu(g)h$. Let 
$$
\cL_{j}(g)=\textbf{L}_j(g\bar h_j)/\la_j \bar h_{j+1}
$$
and let $m_j=\mu_j$ be the sequential Gibbs measures corresponding to the potentials $(\phi_j)$, that is $\mu_j=\bar h_j\nu_j$. 
Then $\cL_j$ is the dual of $T_j$ with respect to  $\mu_j$. Moreover, $(T_j)_*\mu_j=\mu_{j+1}$ and $\cL_j\textbf{1}=\textbf{1}$.


 


\begin{proposition}
 The operators $\cL_j$ obey conditions (LY1), (LY2) and (SC).   
\end{proposition}
\begin{proof}
First, the uniform boundedness (LY1) of the operators $\cL_j$ follows from the properties of the non-normalized RPF triplets (or from \eqref{rppt}). Second,  the Lasota--Yorke inequality (LY2) was obtained in \cite[Lemma 5.12.2]{HK}. Note that in \cite[Lemma 5.12.2]{HK} the weak norm $\|\cdot\|_{L^1}$ is replaced with the (weak) norm 
$\|\cdot\|_\infty$. However, we have $\|g\|_\infty\leq \ve_r v(g)+C_r\|g\|_{L^1}$ with $\ve_r\to 0$ as $r\to 0$. Using this we have that the Lasota--Yorke inequality with respect to $(v_j(\cdot),\|\cdot\|_{L^1})$ is equivalent to the 
Lasota--Yorke inequality with respect to $(v_j(\cdot),\|\cdot\|_{\infty})$. 

Third,
in   \cite{Nonlin} we proved that there are constants $C>0$ and $\del\in(0,1)$ such that for all $j,n$ and a H\"older continuous function $h:X_j\to\bbR$,
\begin{equation}\label{rppt}
   \|\cL_j^n h- m_j(h) \textbf{1}\|_{\al}  \leq C\del^n 
\end{equation}
 and so  \eqref{CR RPF} holds. 
Since $\cL_j\textbf{1}\!\!=\!\!\textbf{1}$ 
we have $\DS \inf_{j,n}\inf\cL_j^n\textbf{1}\!\!=\!\!\del_0\!\!=\!\!1$. Thus  we can repeat the  arguments from the previous section verbatim with time dependent $v$ and $m$ and obtain (SC).
\end{proof}

\section{Moment estimates}\label{Sec Mom}
Let 
$
\DS T_j^n=T_{j+n-1}\circ\cdots\circ T_{j+1}\circ T_j.
$
Consider a sequence of real valued functions $f_k\in B_k$, $k\geq 0$ such that 
$\DS \sup_k\|f_k\|_{BV}<\infty$. Our goal is to obtain limit theorems for
the sequence 
$\DS S_n=\sum_{j=0}^{n-1}f_j\circ T_0^j$
considered as random variables on the probability space $(X,\cB,m_0)$. It what follows it will be convenient to consider $(f_j)$ as  two sided sequence by setting $f_j=0$ for $j<0$.

\subsection{The pulled back measures and their transfer operators}
\label{pull}
Consider the sequence of measures $\tilde m_j=(T_0^j)_*m_0$ on $(X_j,\cB_j)$. Then $\tilde m_0=m_0$ and for all $j>0$ and each bounded measurable function $G:X_j\to\bbC$, 
$$
\tilde m_j(G)=m_0(G\circ T_0^j)=m_j(G\cdot\cL_0^j\textbf{1}).
$$
\begin{remark}
Note that in the case when $m_j=\mu_j$ is an equivariant sequence (i.e. $(T_j)_*\mu_j=\mu_{j+1}$) then $\tilde m_j=\mu_j$.
 \end{remark}
Let the operator $\tilde\cL_j$ be the dual of $T_j$ with respect to the measure $\tilde m_j$ and $\tilde m_{j+1}$
defined by the duality relation
\begin{equation}
\label{DefTL}
\int f\cdot (\tilde\cL_j g)d\tilde m_{j+1}=\int (f\circ T_j)\cdot g\cdot d\tilde m_j
\end{equation} 
for all functions such that both integrals are well defined. Then $\tilde \cL_j\textbf{1}=\textbf{1}$ because $(T_j)_*\tilde m_j=\tilde m_{j+1}$. Let 
$$
\tilde \cL_j^n=\tilde\cL_{j+n-1}\circ\cdots\circ\tilde\cL_{j+1}\circ\tilde\cL_j.
$$

\begin{proposition}\label{Tilde prop}
There are constants $C_0>0$, $J\in\bbN$ and $\del_0\in(0,1)$ such that for all $j\geq 0$ and $n\geq 1$ such that $j+n\geq J$ we have 
$$
\|\tilde\cL_j^n(\cdot)-\tilde m_j(\cdot)\mathbf{1} \|_{BV}\leq C_0\del_0^n.
$$
\end{proposition}
\begin{proof}
Recall that  the functions $h_j$ in Theorem \ref{RPF 0} satisfy 
$\DS \eta:=\inf_j\text{ess-inf }h_j\!\!>\!\!0$. Thus by \eqref{Exp Conv Main} we see that $\text{ess-inf }\cL_j^n\textbf{1}\geq \eta-C\del^n$. Consequently, if we take $\del_0=\frac12\eta$ and $J\in\bbN$ such that $C\del^{J}\leq \frac12\eta$ we see that
$\text{ess-inf }\cL_0^m\textbf{1}\geq \del_0$ for all $m\geq J$. 
Next, a direct calculation shows that 
$\DS
\tilde\cL_j^n g=\frac{\cL_j^n(g\cdot \cL_0^j\textbf{1})}{\cL_0^{j+n}\textbf{1}}.
$
Now the proposition follows from Theorem \ref{RPF 0}, noting that 
$\|\cL_0^m \mathbf 1-h_m\|_{BV}$ decays exponentially fast and using that,
 due to (V6),
 $\DS \sup_j\|1/h_j\|_{BV}<\infty$.  
\end{proof}

\subsection{A martingale coboundary representation}
\label{SSMartCob}

\begin{lemma}\label{Mart lemm}
For every sequence of functions $f_j\in B_j$ such that $\DS \sup_{j}\|f_j\|_{BV}<\infty$ there are functions $M_j=M_j(f)\in B_j$ and $u_j=u_j(f)\in B_j$ such
that
\begin{equation}
\label{MartCob}
\tilde f_j:=f_j-m_0(f_j\circ T_0^j)=M_{j}+u_{j+1}\circ T_j-u_j,\,\,\, j\geq J.
\end{equation}
Moreover, $\DS \sup_j\|u_j\|_{BV}<\infty$ and $(M_{j}\circ T_0^j)_j$  is a reverse martingale difference  with respect to the reverse filtration $\cA_j=(T_0^j)^{-1}\cB_j$ (on the probability space $(X_0,\cB_0,m_0)$). 
Furthermore, $\DS \sup_j\|u_j\|_{BV}$ is bounded above by a constant which depends only on the constants $C$ and $\del$ from Theorem \ref{RPF 0}  and on $\DS \sup_j\|f_j\|_{BV}$.
\end{lemma}
\begin{remark}
    Note that by (V8) we also get that $\DS \sup_j\|M_j\|_{BV}<\infty$.
\end{remark}
\begin{proof}[Proof of Lemma \ref{Mart lemm}]
Define $\tilde \cL_j=0$ for $j<0$ and $\tilde f_j=0$ for $j<J$. Set 
\begin{equation}
\label{DefUj}
u_j=\sum_{k=1}^\infty\tilde \cL_{j-k}^k\tilde f_{j-k}=\sum_{k=1}^{j}\tilde\cL_{j-k}^k\tilde f_{j-k}.
\end{equation}

Since $\tilde m_s(\tilde f_s)=0$ for all $s\geq 0$ by Proposition \ref{Tilde prop}, $u_j\in B_j$ and 
$\DS \sup_j\|u_j\|_{BV}<\infty$. Set $M_{j}=\tilde f_j+u_j-u_{j+1}\circ T_j$. It remains to show that $M_{j}\circ T_0^j$ is indeed a reverse martingale difference. To prove that we notice that 
$\bbE[M_{j}\circ T_0^j|\cA_{j+1}]= \tilde \cL_j(M_j)\circ T_0^{j+1}$. On the other hand, a direct calculation  using \eqref{DefUj} shows that  $\tilde \cL_j(M_{j})=0$.
\end{proof}

\subsection{On the divergence of the variance}

The first step in proving a central limit theorem is to show that the individual summands are negligible in comparison with the variance of the sum.  In particular, we need to know when the variance is bounded.
 In this section we prove the following result.

\begin{theorem}\label{Var them}
The following conditions are equivalent.
\vskip0.2cm
(1) $\DS \liminf_{n\to\infty}\text{Var}_{m_0}(S_n)<\infty$.
\vskip0.2cm
(2) $\DS \sup_{n\in\bbN}\text{Var}_{m_0}(S_n)<\infty$

\vskip0.2cm
(3) We can write 
$\DS
f_j=m_0(f_j\circ T_0^j)+M_j+u_{j+1}\circ T_j-u_j
$
with
$u_j,M_j\in B_j$ such that $M_j\circ T_0^j$ is a reverse martingale difference on $(X_0,\cB_0, m_0)$ with respect to the reverse filtration $T_0^{-j}\cB_j$,
$\DS  \sup_j\|u_j\|_{BV}<\infty,$\;
$\DS\sup_j\|M_j\|_{BV}\!\!<\!\!\infty$, and $\DS \sum_{j}\text{Var}_{m_0}(M_j\circ T_0^j)\!\!<\!\!\infty.$


\end{theorem}

\begin{proof}
First, it is enough to prove the theorem when $m_0(f_j\circ T_0^j)=0$ for all $j$. In this case,
by Lemma   \ref{Mart lemm} we can write
$$
f_j=m_0(f_j\circ T_0^j)+M_j+u_{j+1}\circ T_j-u_j=M_j+u_{j+1}\circ T_j-u_j
$$ 
with $u_j$ and $M_j$ like in (3), except that in general the sum of the variances of $M_j$ might not converge. Notice now that 
\begin{equation}\label{M apprx}
\|S_n-S_nM\|_{L^2(m_0)}\leq \|S_n-S_nM\|_{L^\infty(m_0)}\leq 2\sup_{j}\|u_j\|_{L^\infty(m_j)}:=U<\infty
\end{equation}
where $\DS S_nM=\sum_{j=0}^{n-1}M_j\circ T_0^j$.  

Now assume (1), and let $n_k$ be an increasing sequence such that $n_k\to\infty$ and $\sig_{n_k}=\|S_{n_k}\|_{L^2}\leq C$ for some constant $C>0$. Then by \eqref{M apprx}, 
$\DS
\|S_{n_k}M\|_{L^2(m_0)}^2\leq (C+U)^2<\infty.
$
However, since $M_j\circ T_0^j$ is a reverse martingale, we have
$$
\sum_{j=0}^{n_k-1}\text{Var}_{m_0}(M_j\circ T_0^j)=\|S_{n_k}M\|_{L^2(m_0)}^2\leq (C+U)^2.
$$
Now, since  $\DS V_n:=\|S_nM\|_{L^2}^2=\sum_{j=0}^{n-1}\text{Var}_{m_0}(M_j\circ T_0^j)$ is increasing we conclude that the summability condition in (3) holds. This shows that (1) implies (3). 

Next, it is clear that (2) implies (1). 
Thus, to complete the proof of the theorem it is enough to show that (3) implies (2), but this also follows from \eqref{M apprx} since the latter yields
$\DS
\|S_{n}\|_{L^2(m_0)}^2\leq (V_n+U)^2<\infty.
$
\end{proof}
 
\subsection{Quadratic variation and moment estimates}
Recall that the (unconditioned)  quadratic variation difference of the reverse  martingale difference $M_{j}\circ T_0^j$ is given by 
$\DS 
q_j(M):=M_{j}^2\circ T_0^j.
$
Henceforth we denote $Q_j=M_j^2$ and let
$$
S_{j,n}=S_{j,n}f=\sum_{k=j}^{j+n-1}f_k \circ T_j^k, \quad
\bar S_{j,n}=S_{j,n}-\tilde m_j(S_{j,n})=S_{j,n}-m_0(S_{j,n}\circ T_0^j).
$$
$S_{j,n}M$ and $S_{j,n}Q$ are defined similarly.

\begin{proposition}\label{Thm 4.1'}
There is a constant $C$ which depends only on  the constants from Theorem \ref{RPF 0}  and on $\DS \sup_j\|f_j\|_{BV}$ such that for all $j\geq J$ (where $J$ comes from Proposition \ref{Tilde prop}) we have
$\DS
\text{Var}_{m_0}(S_{j,n} Q)\leq C(1+\text{Var}(S_{j,n} f)).
$
\end{proposition}
\begin{proof}

First, to simplify the notation let us assume that $j=J=0$. Denote 
$$
g_j=Q_{j}-\tilde m_j(Q_{j})=Q_j-\tilde m_0(Q_{j}\circ T_0^j)
$$ 
and $\DS S_{n}g=\sum_{j=0}^{n-1}g_j\circ T_0^j$. 
 The argument below is similar to 
 the first part of the proof of \cite[Theorem 4.1]{CR}, but we provide the details  to make our paper self contained.
 By Proposition \ref{Tilde prop} (using that $\tilde m_0=m_0$) we have
$$
\bbE_{m_0}[(S_ng)^2] \leq 2
\sum_{0\leq \ell <n}\sum_{0\leq k\leq \ell}\left| \tilde m_0\big((g_k\circ T_0^k)\cdot (g_\ell\circ T_0^\ell)\big)\right|
=2
\sum_{0\leq \ell <n}\sum_{0\leq k\leq \ell} \left| \tilde m_{k}(g_k \cdot \tilde\cL_{k}^{\ell-k}g_\ell)\right|
$$
$$
\leq C_0\sum_{0\leq \ell <n}\sum_{0\leq k\leq \ell}\tilde m_{k}(|g_k|)\| g_\ell\|_{BV}\del_0^{\ell-k}= C_0\sum_{0\leq k\leq n}\tilde m_{k}(|g_k|)\left(\sum_{k\leq \ell <n}\|g_\ell\|_{BV}\del_0^{\ell-k}\right)
$$
$$
\leq c_0\sum_{0\leq k\leq n}\tilde m_{k}(|g_k|)= c_0\sum_{0\leq k\leq n} m_{0}(|g_k\circ T_0^k|)\leq 2c_0\sum_{0\leq k\leq n} m_{0}(Q_k\circ T_0^k)
$$
for some constant $c_0$ (the first inequality of the last line   uses that $\DS \sup_j\|g_j\|_{BV}<\infty$). Here $\tilde \cL_j$ are the transfer operators
 defined by \eqref{DefTL}.
Observe that 
$
 \DS m_{0}(Q_{k}\circ T_0^k)\!\!=\!\!m_0((M_{k}\circ T_0^k)^2)
$
and, because of the orthogonality property,
$\DS
\sum_{0\leq k<n}m_0((M_{k}\circ T_0^k)^2)=\text{Var}_{m_0}(S_n M).
$
 Now the result follows from \eqref{M apprx}.
\end{proof}



\begin{proof}[Proof of Proposition \ref{Mom prop}]
It is enough to prove the theorem  for $j\geq J$, since to get the result for 
$0\leq j<J$ we can just take  $C$ large enough.

To simply the notation, we will only  prove the theorem when $j=J=0$, the proof when $j\geq J>0$ is similar. 

Notice that it is enough to prove the claim for $p=2^m$ for all $m$.  Moreover, by replacing $f_j$  with $f_j-\mu_0(f_j\circ T_0^j)$ we can and will assume that $\mu_0(S_{n})=0$ for all $n$.

We use induction on $m$.   For  $m=1$ the result  is trivial. 
Suppose that the statement is true for some $m\geq 1$. In order to estimate $\| S_nf\|_{2^{m+1}}$ we first use that\footnote{Where 
$\DS  S_n M=\sum_{j=0}^{n-1}M_j\circ T_0^j$, $M_j=M_j(f)$.} 
$$
\|S_nf\|_{2^{m+1}}\leq C+\|S_n M\|_{2^{m+1}}
$$
for some constant $C$ which depends only on the constants $C$ and $\delta$ from Theorem~\ref{RPF 0} and on $\DS \|f\|:=\sup_j\|f_j\|_{BV}$ 
 since
$\|S_nf-S_nM\|_{L^\infty}$ is bounded in $n$. So it suffices to show that
 \begin{equation}\label{en}
\|S_n M\|_{2^{m+1}}\leq C(1+\|S_n f\|_{2})
\end{equation}
 for an appropriate constant $C$.
 
 \eqref{en} follows
from a version of Burkholder's inequality for martingales (see \cite[Theorem 2.12]{PelB}). Let $\fd_1,....,\fd_n$ be a martingale difference with respect to a filtration $(\cF_j)_{j=1}^n$ on a probability space. Let $D_n=\fd_1+\fd_2+...+\fd_n$ and
$E_n=\fd_1^2+\fd_2^2+...+\fd_n^2$.
Then, for every $p\geq 2$ there are constants $c_p,C_p>0$ depending only on $p$ such that 
\begin{equation}\label{Burk}
 c_p\|E_n\|_{p/2}^{1/2}\leq \|D_n\|_{p}\leq C_p\|E_n\|_{p/2}^{1/2}. 
\end{equation}
Applying \eqref{Burk}  with   the (reverse) martingale difference $M_j\circ T_0^j$ we see that 
\begin{equation}\label{Burk1}
\|S_n M\|_{2^{m+1}}\leq a_m\|S_nQ\|_{2^m}^{1/2}
\end{equation}
where $S_n Q=S_{0,n}Q$ and
$a_m$ 
depends only on $m$. Applying the induction hypothesis with the sequence of functions $\tilde Q_j=Q_j-m_0(Q_j\circ T_0^j)$
 we see that there is a constant $R_m>0$ depending only of $m$ and the constants in the formulation of Proposition \ref{Mom prop} such that 
 $$
 \|S_n\tilde Q\|_{2^m}\leq R_m(1+\|S_n\tilde Q\|_{2}).
$$
 Since $\bbE[S_nQ]=\text{Var}(S_n M)$, 
 Proposition \ref{Thm 4.1'} 
 gives
$$
\|S_n Q\|_{2^m}\leq  \|S_n\tilde Q\|_{2^m}+\bbE[S_nQ]\leq R_m\left(1+C(1+\text{Var}(S_n
f))\right)+ \text{Var}(S_n M)
$$
$$
\leq R'_m(1+\text{Var}(S_n f))+\text{Var}(S_n M)
$$
for some other constant $R_m'$.
Using that  $\DS \sup_n\|S_nf-S_nM\|_{L^\infty}<\infty$ we see that 
$\DS
\text{Var}(S_n M)\leq C
\left(1+\text{Var}(S_n f)\right)
$
for some constant $C>0$. Thus, there is a constant $R_m''>0$ such that  
$$
 \|S_n\tilde Q\|_{2^m}\leq R_m''(1+\text{Var}(S_n f)).
$$
Now \eqref{en} follows from \eqref{Burk1}, completing the proof of the proposition.
\end{proof}

\section{Logarithmic growth conditions}\label{Sec Log}
In this section we prove the following result.
\begin{proposition}\label{Growth Prop}
Let $f_j:X_j\to\bbR$ be a sequence of functions such that 
$\DS \sup_j\|f_j\|_{BV}\!\!<\!\!\infty$ and $\tilde m_j(f_j)=0$.
Suppose $\sig_n^2=\text{Var}_{m_0}(S_n)\to\infty$.
Let $\Lambda_n(t)=\ln m_0(e^{it S_n f/\sig_n})$. Then for every $k\geq3$ there are constants $\del_k>0$ and $C_k>0$ such that 
$$
\sup_{t\in[-\del_k\sig_n,\del_k\sig_n]}|\Lambda^{(k)}_n(t)|\leq C_k\sig_n^{-(k-2)}
$$
where $\Lambda_n^{(k)}$ is the $k$-th derivative of the function $\Lambda_n$.
\end{proposition}

\noindent
The proof 
uses
the complex Ruelle-Perron-Frobinuous theorem,  see Theorem \ref{CMPLX RPF} below.

\subsection{Sequential   complex RPF theorem}\label{RPF sec}

Take a sequence of real valued functions $f_j\in B_j$ and $\DS \|f\|=\sup_j\|f_j\|_{BV}<\infty$. Consider 
the operators 
$\cL_{j,z}(h)\!\!=\!\!\cL_j(he^{z f_j})$  (where $z\in\bbC$). We need the following result.

\begin{lemma}
\,

(i)  $\DS \sup_j\|\cL_{j,z}\|_{BV}\leq C(z)$ for some continuous function $C(z)$ (of exponential order in $|z|$).
\vskip0.1cm

\noindent
(ii) The map $z\!\!\to\!\!\cL_{j,z}$ is analytic  and its $k$-order derivatives $\cL_{j,t,k}$ satisfy 
$\DS
\sup_j\|\cL_{j,t,k}\|_{BV}\leq C_k(z)
$
for some continuous function  $C_k(\cdot)$  (of exponential order in $|z|$).
\end{lemma}

\begin{proof}
Since
$\DS
\|e^{zf_j}\|_{BV}\leq \sum_{k}\frac{|z|^k\|f_j^k\|_{BV}}{k!}
$
and by (V5), $\|f_j^k\|_{BV}\leq C^k\|f_j\|_{BV}^k$ for some constant $C$, we get  
$\DS \sup_j\|e^{zf_j}\|_{BV}\leq e^{c|z|}
$
for some constant $c>0$.

The analyticity of the operators follows from the analytiticty of the map $z\to e^{zf_j}$, and the estimate on the norms of the derivatives is obtained similarly to part (i).
\end{proof}

The following result is a consequence of the more general perturbation theorem 
(Theorem~\ref{PertThm} in  Appendix \ref{AppPert}).

\begin{theorem}\label{CMPLX RPF}
There is a constant $r_0>0$ such that for every complex number with $|z|\leq r_0$
there are sequences $\la(z)=(\la_j(z)), \la_j(z)\in\bbC\setminus\{0\}$, 
$h^{(z)}=(h_j^{(z)}), h_j^{(z)}\in B_j$, $m^{(z)}=(m_j^{(z)}), m_j^{(z)}\in B_j^*$
with the following properties.
\vskip0.1cm
(i) The maps $z\to \la_j(z)$, $z\to h_j^{(z)}$ and $z\to m_j^{(z)}$ are analytic in $z$ 
 and their norms\footnote{Here if $\gamma$ is a map from $\{|z|\leq r_0\}$ to a Banach space
we take $\DS \|\gamma\|=\sup_{|z|\leq r_0} \|\gamma(z)\|.$}
  are bounded uniformly in $j$.
\vskip0.1cm

(ii) We have 
$\la_j(0)=1$, $m_j^{(0)}=m_j$, $h_j^{(0)}=h_j$,
 $m_j^{(z)}(\textbf{1})=m_j^{(z)}(h_j^{(z)})=1$ and
$$
\cL_{j,z}h_j^{(z)}=\la_j(z)h_{j+1}^{(z)}, \quad (\cL_{j,z})^*m_{j+1}^{(z)}=\la_j(z)m_{j}^{(z)}.
$$
\vskip0.1cm

(iii) There are $\del_3\in(0,1)$ and $C_3>0$ such that 
$$
\left\|\cL_{j,z}^n-\la_{j,n}(z)m_j^{(z)}(\cdot)h_{j+n}^{(z)}\right\|\leq C_3\del_3^n
$$
where 
$
\cL_{j,z}^n=\cL_{j+n-1,z}\circ\cdots\circ\cL_{j+1,z}\circ\cL_{j,z}
$
and $\la_{j,n}(z)=\la_{j+n-1}(z)\cdots\la_{j+1}(z)\la_j(z)$.
 \end{theorem}
 
\subsection{Proof of Proposition \ref{Growth Prop}}

Let $\la_j(z)$ be like in Theorem \ref{CMPLX RPF}, and
let $\Pi_j(z)$ be an analytic branch of $\ln \la_j(z)$ so that $\Pi_j(0)=1$.
\begin{lemma}\label{LLL}
There are $\ve_0>0$ and $A_0>0$ such that for every complex number $z$ with $|z|\leq \ve_0$ and all $j,n$ with $j+n\geq J$ we have
$$
\left|\ln \bbE_{m_0}[e^{z(S_{n+j} f-S_{j}f)}]-\sum_{k=j}^{j+n-1}\Pi_j(z)\right|\leq A_0.
$$
\end{lemma}

\begin{proof}
Since $m_0=\tilde m_0$ we have
$\DS
 \bbE_{m_0 }[e^{z(S_{n+j} f-S_{j}f)}]=\tilde m_{j+n}(\tilde\cL_{j,z}^n\textbf{1})
$
where 
\begin{equation}
\label{TildeL}
\tilde\cL_{j,z}^n(g)=\tilde\cL_j^n(e^{zS_{j,n}f} g)=\frac{\cL_{j,z}^n(g\cL_0^j\textbf{1})}{\cL_0^{j+n}\textbf{1}}.
\end{equation}
Using 
 Theorem \ref{CMPLX RPF}(iii),  and the facts
 that $m_j(\cL_0^j\textbf{1})=1$, $\text{essinf}(\cL_0^{j+n} \textbf{1})\geq c_0>0$ and 
 $\DS \sup_k\|\cL_0^k\|_{BV}<\infty$ we get
$$   
\tilde m_{j+n}(\tilde\cL_{j,z}^n\textbf{1})=\la_{j,n}(z)U_{j,n}(z)+
O(\del_3^n).
$$
where $\del_3\in(0,1)$ and 
$\DS
U_{j,n}(z)=m_j^{(z)}(\cL_0^j\textbf{1})\tilde m_{j+n}(H_{j+n,z}),$\hskip2mm
$\DS H_{j+n,z}=\frac{h_{j+n}^{(z)}}{\cL_0^{j+n}\textbf{1}}.
$
Now, let $G(z)=\tilde m_{j+n}(\tilde\cL_{j,z}^n\textbf{1})-\la_{j,n}(z)U_{j,n}(z)$. Then $G$ is is analytic. 
Since
$$
\tilde m_{j+n}(\tilde\cL_{j,0}^n\textbf{1})=\la_{j,n}(0)U_{j,n}(0)=1
$$
we have $G(0)\!\!=\!\!0$. Applying the maximal modulus principle to the function $\frac{G(z)}{z}$ we get
$$   
\tilde m_{j+n}(\tilde\cL_{j,z}^n\textbf{1})=\la_{j,n}(z)U_{j,n}(z)+
zO(\del_3^n)
$$
namely, the error term is $\DS O\left(z\del_3^n\right)$ and not only $O(\del_3^n)$.
Next
$$
\tilde m_{j+n}(H_{j+n,z})=m_{j+n}(H_{j+n,z}\cL_0^{j+n}\textbf{1})=m_{j+n}(h_{j+n}^{(z)}).
$$
Since $\la_j(0)=1$ and $\la_j(z)$ are uniformly bounded and
analytic
in $z$, we can write the above in the following form
$$
\tilde m_{j+n}(\tilde \cL_{j,z}^n\textbf{1})
=\la_{j,n}(z)\left(U_{j,n}(z)+zO(\del_4^n)\right),\quad \del_4\in(0,1).
$$
The functions $U_{j,n}(z)$ are analytic, uniformly bounded in $z$ and they take the value one at $z=0$. Thus, we can take the logarithm of both sides to conclude that, if $|z|$ is small enough then
$$
\ln \bbE_{m_0 }[e^{z(S_{n+j} f-S_{j}f)}]=\ln \la_{j,n}(z)+V_{j,n}(z)
$$
where $V_{j,n}(z)=\ln U_{j,n}(z)=\ln (1+O(z)+zO(\del_4^n))=O(1)$.
\end{proof}

Using the Cauchy integral formula we derive the following result.
\begin{corollary}\label{Cor z}
Let $\tilde\Lambda_{j,n}(z)=\ln  \bbE_{m_0}[e^{z(S_{j+n} f-S_j f)}]$.
Then there exists  $\ve>0$ such that  for every $s\geq 1$ there 
is a constant $C_s>0$ such that for every complex $z$ with $|z|\leq \ve$ and all $j,n$ with $j+n\geq J$ we have
 \begin{equation}\label{Cor z est}
 \left|\tilde\Lambda_{j,n}^{(s)}(z)-\sum_{k=j}^{j+n-1}\Pi_k^{(s)}(z)\right|\leq C_s
 \end{equation}
 where $g^{(s)}$ denotes the $s$-th derivative of a function $g$.

\end{corollary}
\begin{remark}
Clearly, \eqref{Cor z est} also holds  when $j+n<J$ since then the number of summands is uniformly bounded.
\end{remark}

Set 
\begin{equation}\label{pi def}
\Pi_{j,n}(z)=\sum_{k=j}^{j+n-1}\Pi_k(z).
\end{equation}

\begin{lemma}\label{L35}
Let $B$ be a constant and let $k\geq 2$. Then if $B$ is
sufficiently large there are constants $D$ and $r_0$ depending only on $B$ and $k$ so that
for every $t\in[-r_0,r_0]$ and each $j,n$ such that $B\leq \text{Var}(S_{j+n}f-S_j f)\leq 2B$ we have
$$
|\Pi_{j,n}^{(k)}(it)|\leq D.
$$
\end{lemma}
\begin{proof}
Applying \cite[Lemma 43]{DH} with $S=S_{j+n}f-S_j f$ we see that  there is $r=r(B)$ such that if $t\in[-r,r]$ then
$$
|\tilde\Lambda_{j,n}^{(k)}(it)|\leq D_k\bbE_{m_0}[|S|^k]
$$
for some constant $D_k$ which depends only on $k$. Using also Corollary \ref{Cor z} we derive that 
$$
|\Pi_{j,n}^{(k)}(it)|\leq D_k\bbE_{m_0}[|S|^k].
$$
Next, by Proposition \ref{Mom prop} we have 
$\DS
\bbE_{m_0}[|S|^k]\leq C(k,B).
$
These estimates prove the lemma.
\end{proof}

\begin{proof}[Proof of Proposition \ref{Growth Prop}]
Fix some $k\geq3$.
Since $\sig_n=\|S_n\|_{L^2(m_0)}\to\infty$, using the martingale coboundary representation from Lemma \ref{Mart lemm}, given $B>0$ large enough we can decompose 
$\{0,...,n\!-\!1\!\}$ into a disjoint union of intervals $I_1,...,I_{k_n}$ in $\bbZ$ so that $I_j$ is to the left of $I_{j+1}$  and 
\begin{equation}\label{Var block}
B\leq \text{Var}_{m_0}(S_{I_j})\leq 2B
\end{equation}
where $\DS S_I=\sum_{j\in I}f_j\circ T_0^j$ for every interval $I$. Now, by Lemma \ref{Mart lemm} there is a constant\footnote{We can take $\DS A=2\sup_j\|u_j\|_{L^\infty}$,
 where $u_j$ are the coboundaries  from \eqref{MartCob}
.} $A>0$ independent of $B$ such that 
$\DS
\left|\|S_nf\|_{L^2}-\left(\sum_{k=J}^{n-1}\text{var}_{m_0}(M_k\circ T_0^k)\right)^{1/2}\right|\leq A
$
and for each $j>J$ we have 
$\DS
\left|\|S_{I_j}\|_{L^2}-\left(\sum_{k\in I_j}\text{var}_{m_0}( M_k\circ T_0^k)\right)^{1/2}\right|\leq A.
$

Hence, if we also assume that $B>(4A)^2$ then it follows that 
\begin{equation}\label{kn prop}
C_1\leq k_n/\sig_n^2\leq C_2
\end{equation}
for some constants $C_1,C_2>0$ which depend only on $B$.
Next, let $\DS \Pi_{I}(z)=\sum_{k\in I}\Pi_k(z)$. Then
by Lemma \ref{L35} there are constants $r_k>0$ and $D_k$ such that
$$
\sup_{j}\sup_{t\in[-r_k,r_k]}\left|\Pi_{I_j}^{(k)}(it)\right|\leq D_k.
$$
Hence,
\begin{equation}\label{pii bound}
\sup_{t\in[-r_k,r_k]}\left|\Pi_{0,n}^{(k)}(it)\right|\leq D_k k_n\leq C_2
D_k\sig_n^2.
\end{equation}
Combining this  with Corollary \ref{Cor z} and taking into account that $\sig_n\to\infty$ we see that 
$$
\sup_{t\in[-r_k,r_k]}\left|\tilde\Lambda_n^{(k)}(it)\right|\leq \tilde D_k\sig_n^2
$$
for some constant $\tilde D_k$, and the proof of the proposition is complete. 
\end{proof}

\section{Proof of the main results}\label{BE sec}
\subsection{Proof of Theorems \ref{BE} and \ref{ThWass}}
Let $\Lambda_n(t)=m_0(e^{it S_n/\sig_n})$. 
 Applying Theorems 5 and 9, and Corollary 11 from \cite{H} to $W_n=S_n/\sig_n$,
we see that
in order to prove Theorems \ref{BE} and \ref{ThWass} 
 it is enough to show that for every $u\geq 3$ there are constants $\del_u,C_u>0$ such that the function $\Lambda_n(t)$ is well defined, differentiable $u$ times on $[-\del_u\sig_n,\del_u\sig_n]$ and
$\DS \sup_{|t|\leq \sig_n\del_u} \!\! |\Lambda_n^{(u)}(t)|\!\!\leq\!\! C_u\sig_n^{-(u-2)}.
$
However, this is exactly Proposition \ref{Growth Prop}.

\subsection{Proof of Theorem \ref{ASIP}}
Let $A_k=S_{I_k}$, where $I_k$ were introduced in the proof of Proposition \ref{Growth Prop}.
Then, since $\DS \sup_k\|A_k\|_{L^2}<\infty$, we obtain from Proposition \ref{Mom prop} that for all $p$ we have 
$\DS \sup_k\|A_k\|_{L^p}<\infty$.
The ASIP will follow from \cite[Theorem 2.1]{DDH}, applied with the real-valued sequence $A_k=S_{I_k}$, and with an arbitrarily large $p$. In what follows we will verify   the conditions of   \cite[Theorem 2.1]{DDH}.

The verification of condition (2.1) in  
 \cite[Theorem 2.1]{DDH} is similar to \cite[\S 4.4]{DDH}, with some modifications. For readers' convenience we will provide the proof.
We have to show that there exists $\varepsilon_0\!>\!\!0$ and $C,c\!>\!\!0$ such that for all $k,n,m\!\!\in\! \bbN$, $b_1\!<\!b_2\!<\!\! \dots \!\!<\!b_{n+m+k}$,  and $t_1,\ldots ,t_{n+m}\in\bbR$ with $|t_j|\leq\varepsilon_0$, we have 
\begin{equation}\label{Gouzel}
\Big|\mathbb E\big(e^{i\sum_{j=1}^nt_j(\sum_{\ell=b_j}^{b_{j+1}-1}A_\ell)+i\sum_{j=n+1}^{n+m}t_j(\sum_{\ell=b_j+k}^{b_{j+1}+k-1}A_\ell)}\big)
\end{equation}
$$
-\mathbb E\big(e^{i\sum_{j=1}^nt_j(\sum_{\ell=b_j}^{b_{j+1}-1}A_\ell)}\big)\cdot\mathbb E\big(e^{i\sum_{j=n+1}^{n+m}t_j(\sum_{\ell=b_j+k}^{b_{j+1}+k-1}A_\ell)}\big)\Big|
$$
$$
\leq C(1+\max|b_{j+1}-b_j|)^{C(n+m)}e^{-ck}. 
$$
 Recall
 the definition \eqref{TildeL} of the operators $\tilde\cL_{j,n}^n$.
\begin{lemma}
 There is a constant $\ve_0>0$ such that 
\begin{equation}\label{Bound}
 \sup_{j}\sup_{n}\sup_{|t|\leq \ve_0}\|  \tilde \cL_{j,it}^n\|<\infty.  
\end{equation}
\end{lemma}
\begin{proof}
It follows from the definition of the operators $\tilde \cL_j$ and the uniform bounds (from above in the norm and below) on $\cL_0^n\textbf{\textbf{1}}$ that there is a constant $C>0$ such that for all real $t$,
$$
\|\tilde \cL_{j,it}^n\|\leq  C\| \cL_{j,it}^n\|.
$$
Now, by Theorem \ref{CMPLX RPF} 
we see that there  are constant $\ve$ and $C_1>0$ such that for every $j,n$ and $t\in[-\ve,\ve]$ we have 
$\DS 
\| \cL_{j,it}^n\|\leq C_1|\la_{j,n}(it)|=C_1|e^{\Pi_{j,n}(it)}|.
$ 
Indeed, since $\la_j(0)=1$ 
we have $\del_3^n=O(|\la_{j,n}(z)|)$  for $|z|$ small enough.
Now, using the analyticity in $z$ of the term inside the absolute value on the right hand side of 
 in \eqref{Cor z est}, differentiating twice at the point $z=0$ and using the Cauchy integral formula yields 
$$
\sup_{j,n}|\Pi_{j,n}'(0)|<\infty\,\,\text{ and }\,\,\sup_{j,n}|\Pi_{j,n}'(0)-\text{Var}(S_{j,n})|<\infty.
$$
where in the first estimate we have used that $\bbE[S_{j,n}]=0$. Next, 
we claim that 
there exist constants $C_0,\ve_1$ such that if $t\in[-\ve_1,\ve_1]$ then
\begin{equation}\label{cllaim}
|\Pi_{j,n}'''(it)|\leq C_0(1+\text{Var}(S_{j,n})).   
\end{equation}
The proof of \eqref{cllaim} when $\text{Var}(S_{j,n})\geq C$ for a sufficiently large $C$ is similar to the proof of Proposition \ref{Growth Prop}. Indeed, like at the beginning of the proof of  Proposition \ref{Growth Prop} we can decompose the set $\{j,j+1,...,j+n-1\}$ into blocks such that the variance along each block is between $B$ and $2B$ for a sufficiently large constant $B$. This is possible if $C$ is large enough. Repeating the rest of the arguments in the proof of  Proposition \ref{Growth Prop} with $\{j,j+1,...,j+n-1\}$ instead of $\{0,1,2,...,n-1\}$ and with $k=3$ we get the following version of 
\eqref{pii bound}: there is a constant $r_3>0$ such that
$$
\sup_{t\in[-r_3,r_3]}\left|\Pi_{j,n}'''(it)\right|\leq C_2\text{Var}(S_{j,n})\leq C_2(1+\text{Var}(S_{j,n})).
$$

To complete the proof of \eqref{cllaim} it remains to show that for every constant $C>0$ there are constants $C_0,\ve_1$ such that \eqref{cllaim} holds when 
$\text{Var}(S_{j,n})\!\leq\! C$. By applying Corollary~\ref{Cor z} this task is reduced  to showing that for every $C>0$ there are constants $C_0'>0$ and $\ve_1>0$ such that 
\begin{equation}\label{cllaim1}
|\tilde\Lambda_{j,n}'''(it)|\leq C_0'(1+\text{Var}(S_{j,n}))
\end{equation}
if $|t|\leq \ve_1$ and $\text{Var}(S_{j,n})\leq C$.  In order to prove \eqref{cllaim1}, we apply  \cite[Lemma 43]{DH} with $S=S_{j+n}f-S_j f$  and $k=3$ in order to find $\ve_1>0$ and $C_1>0$ (which depend on $C$) such that 
$$
\sup_{|t|\leq \ve_1}|\tilde\Lambda_{j,n}'''(it)|\leq C_1\bbE_{m_0}[|S_{j,n}|^3].
$$
Now \eqref{cllaim1} follows by Proposition \ref{Mom prop}, taking into account that 
$$\|S_{j,n}\|_{L^2}^3=(\text{Var}(S_{j,n}))^{3/2}\leq C^{3/2}\leq 
C^{3/2}(1+\text{Var}(S_{j,n})).$$

\noindent
Next, using \eqref{cllaim} together with the Lagrange form for the Taylor remainder of order $2$ of the function $\Pi_{j,n}$ around the origin we conclude that  
$$
\left|\Pi_{j,n}(it)+\frac12t^2(\text{Var}(S_{j,n}))\right|\leq C_1\left(|t|+t^2+|t|^3+|t|^3\text{Var}(S_{j,n})\right).
$$
where $C_1$ is some constant.
Therefore, if $\ve_0$ is small enough and $|t|\leq \ve_0$ then
$\DS
\| \cL_{j,it}^n\|\leq C_1'e^{-\frac14 t^2\text{Var}(S_{j,n})}
$
for some other constant $C_1'$.
In particular, \eqref{Bound} holds.
\end{proof}
 
Next let $Q_j(g)=\tilde m_j(g)\textbf{1}_{j+1}$, where we recall that $\textbf{1}_j$ is the function taking the constant value $1$ on $X_j$ and $\tilde m_j=(T_0^j)_*m_0$. Note that
$$
Q_j^k(g)=Q_{j+k-1}\circ\cdots\circ Q_{j+1}\circ Q_j(g)=\tilde m_j(g)\textbf{1}_{j+k}.
$$
Using these notations, we have
$$
 \mathbb E_{m_0}\big(e^{i\sum_{j=1}^n t_j \cdot (\sum_{\ell=b_j}^{b_{j+1}-1}A_\ell)+i\sum_{j=n+1}^{n+m}t_j \cdot (\sum_{\ell=b_j+k}^{b_{j+1}+k-1}A_\ell)}\big) 
 $$
 $$=\tilde m_{b_{n+m}+k}\left(\tilde\cL_{b_{n+m}+k,it_{n+m}}^{b_{n+m+1}-b_{n+m}}\circ                       \cdots   \circ\tilde\cL_{b_{n+1}+k,it_{n+1}}^{ b_{n+2}-b_{n+1}}\circ    \tilde\cL_{b_{n+1}}^{k}\circ  \tilde\cL_{b_n,it_n}^{b_{n+1}-b_n}\circ   \cdots  \circ\tilde\cL_{b_1,it_1}^{b_2-b_1}\circ\tilde\cL_0^{b_1}\textbf{1}\right) 
 $$
 $$
= \tilde m_{b_{n+m}+k}\left(\tilde\cL_{b_{n+m}+k,it_{n+m}}^{ b_{n+m+1}-b_{n+m}}\circ                       \cdots\circ  \tilde\cL_{b_{n+1}+k,it_{n+1}}^{ b_{n+2}-b_{n+1}}\circ(\tilde\cL_{b_{n+1}}^{k}-Q_{b_{n+1}}^k)\circ  \tilde\cL_{b_n,it_n} ^{b_{n+1}-b_n}\circ   \cdots\circ  \tilde \cL_0^{b_1}\textbf{1}\right)
 $$
$$+\tilde m_{b_{n+m}+k}\left(\tilde\cL_{b_{n+m}+k,it_{n+m}} ^{b_{n+m+1}-b_{n+m}}\circ                       \cdots\circ  \tilde\cL_{b_{n+1}+k,it_{n+1}}^ {b_{n+2}-b_{n+1}}\circ Q_{b_{n+1}}^k\circ \tilde \cL_{b_n,it_n}^{b_{n+1}-b_n}\circ   \ldots\circ   \tilde\cL_0^{b_1}\textbf{1}\right). 
$$
Now \eqref{Gouzel} follows by \eqref{Bound}, the estimates 
$$
\|\tilde \cL_j^k-Q_j^k\|\leq C\del^k
$$
for $\del\in(0,1)$ and $C>0$ and the observation that 
$$
\tilde m_{b_{n+m}+k}\left(\tilde\cL_{b_{n+m}+k,it_{n+m}} ^{b_{n+m+1}-b_{n+m}}\circ                       \cdots\circ  \tilde\cL_{b_{n+1}+k,it_{n+1}}^ {b_{n+2}-b_{n+1}}\circ Q_{b_{n+1}}^k\circ \tilde\cL_{b_n,it_n}^{b_{n+1}-b_n}\circ   \ldots\circ  \tilde\cL_0^{b_1}\textbf{1}\right)
$$
$$
=\mathbb E\big(e^{i\sum_{j=1}^nt_j(\sum_{\ell=b_j}^{b_{j+1}-1}A_\ell)}\big)\cdot\mathbb E\big(e^{i\sum_{j=n+1}^{n+m}t_j(\sum_{\ell=b_j+k}^{b_{j+1}+k-1}A_\ell)}\big).
$$

Next, note that if $B$ from \eqref{Var block} is large enough then condition (2.2) in \cite[Theorem~2.1]{DDH} is satisfied, namely that  the variance of $A_1+...+A_n$ grows at least linearly fast in $n$. 
Indeed, relying on  Lemma \ref{Mart lemm}, the sums $A_a+...+A_b$ can be approximated in $L^\infty$ by an appropriate sum $Z_a+...+Z_b$ of reverse martingale differences, uniformly in $a$ and $b$ and $B$. Taking $B$ large enough we see that the variance of each $Z_j$ is bounded below by a positive constant and so, by the orthogonality property of martingales, the variance of $Z_1+...+Z_n$ grows at least linearly fast in $n$. Thus, the variance of $A_1+...+A_n$ grows at least linearly fast in $n$.

Next, we verify  condition (2.5) in \cite[Theorem 2.1]{DDH}. We need to show that there exist constants $C>0$ and $\del\in(0,1)$ such that for all $j$ and $k$,
\begin{equation}\label{Dec}
|\text{Cov}(A_j,A_{j+k})|\leq C\del^k.    
\end{equation}
We first recall that there are constants $C_0>0$ and $\del\in(0,1)$ such that, with $F_j=f_j\circ T_0^j$ we have 
$\DS
|\text{Cov}(F_j,F_{j+k})|\leq C_0\del^k.
$

In order to prove \eqref{Dec}, let us write $I_j=\{a_j,a_{j+1},...,b_j\}$. Then
$$
|\text{Cov}(A_j,A_{j+k})|=\left|\text{Cov}(S_{I_j}, S_{I_{j+k}})\right|\leq 
\sum_{n\in I_{j}, m\in I_{j+k}}\left|\text{Cov}(F_n,F_{m})\right|
$$
$$
\leq C_0\sum_{n\in I_{j}, m\in I_{j+k}}\del^{m-n}=C_0\del^{a_{j+k}-b_j}\sum_{n\in I_{j}, m\in I_{j+k}}\del^{m-a_{j+k}}\del^{b_j-n}\leq C_0\del^{a_{j+k}-b_j}(1-\del)^{-2}.
$$
To complete the proof of \eqref{Dec}, we note that $a_{j+k}-b_j\geq k$ since there are $k$ blocks between $I_j$ and $I_{j+k}$.

\noindent
We conclude that all the results in \cite[Theorem 2.1]{DDH} 
 are true for the sequence $A_j=S_{I_j}$ and every $p\!\!>\!\!2$. Thus
there is a coupling between the sequence $A_1,A_2,...$ and a sequence $G_1,G_2,...$ of independent centered Gaussian random variables such that for every $\ve>0$,
\begin{equation}\label{asip.0}
\left|\sum_{i=1}^{k}A_i-\sum_{j=1}^k G_j\right|=o(k^{\frac 14+\ve}),\,\,\text{a.s.}
\end{equation}
and  all the properties specified in Theorem \ref{ASIP} hold with the partial sums of the $A_j$'s.

Next, let $D_j=\max\{|S_{b_j}-S_k|: k\in I_j\}$, where we recall $I_j=\{a_j,a_{j+1},...,b_j\}$.  We claim next that for every $p\geq 2$ there is a constant $c_p>0$ such that 
\begin{equation}\label{nnn}
\sup_j\|\cD_j\|_p\leq c_p.  
\end{equation}
Define 
$$
\cR_j=\max\{|S_{b_j}M-S_kM|: k\in I_j\}
$$
where $M=(M_j)$ is the martingale part of $(f_j)$. Then, since $\DS \sup_j\|S_{I_j}M\|_p\leq a<\infty$ for some $a_p>0$, by  Doob's maximal inequality we see that  for all $p>1$ there exists $b_p>0$ such that
$\DS
\sup_j\|\cR_j\|_p\leq b_p.
$

To complete the proof of \eqref{nnn} we note that  
$d_j=\max\{|S_{\ell}-S_\ell M|: \ell\in I_j\}$ satisfy
$\DS
  \sup_j\|d_j\|_\infty<\infty
$
since $(f_j)$ and $(M_j)$ differ by a sequential coboundary (see Lemma \ref{Mart lemm}).\vskip1mm

Now we are ready to complete the proof of Theorem \ref{ASIP}. The following arguments are similar to  \cite[Section 4]{Haf SPL}. First, 
by the Markov inequality we see that for every $\ve>0$ and $p>2$ such that $\ve p>1$ we have 
$$
\bbP(|\cD_j|\geq j^\ve)=\bbP(|\cD_j|^p\geq j^{\ve p})=O(j^{-\ve p}).
$$
By the Borel-Cantelli Lemma $|\cD_j|=O(j^\ve)$, a.s. Next, let
$k_n=\max\{k: b_k\leq n \}$. Then
\begin{equation}\label{eest}
\left|S_n-\sum_{j\leq k_n}A_j\right|\leq \cD_{k_n}=O(k_n^\ve)=O(\sig_n^{2\ve})
\end{equation}
where in the last equality we have used \eqref{kn prop}.
Next, let $F_j=f_j\circ T_0^j$. Let us consider the following sequences $X,Y$ and $Z$ of random vectors: 
$X=(F_1,F_2,...)$, $Y=(Y_1,Y_2,...)$ where $Y_{j}=A_{i}$ if $j=b_i$ and $Y_j=0$ otherwise, and $Z=(Z_1, Z_2,...)$, where $Z_{j}=G_{i}$ if $j=b_i$ and $Z_j=0$ otherwise. 
\begin{lemma}\label{BP L}
There is a probability space on which we can define all three sequences $X,Y$ and $Z$ without changing the  distributions of $(X,Y)$ and $(Y,Z)$.
\end{lemma}
Using this lemma Theorem \ref{ASIP} follows by plugging in $k=k_n$ in \eqref{asip.0} using that $k_n\asymp \sig_n^2$ and applying \eqref{eest} with $\ve<\frac14$.
Lemma \ref{BP L} appeared as \cite[Lemma 4.1]{Haf SPL}, but we include its proof for the sake of completeness. 
\begin{proof}[Proof of Lemma \ref{BP L}]
Iterating the Berkes-Phillip lemma (see \cite[Lemma A.1]{BP}) we see that for each $n $ there is a probability measure $P_n$ on $(\bbR^d\times \bbR^d\times \bbR^d)^n$ which gives raise to a random vector $\{(\tilde X_j,\tilde Y_j,\tilde Z_j), j\leq n\}$ such that $\{(\tilde X_j,\tilde Y_j), j\leq n\}$ and $\{(X_j,Y_j), j\leq n\}$ have the same distribution and  $\{(\tilde Y_j,\tilde Z_j), j\leq n\}$ and $\{(Y_j,Z_j), j\leq n\}$ have the same distributions. Let us extend $P_n$ arbitrarily to a probability measure on $(\bbR^d\times \bbR^d\times \bbR^d)^\bbN$. Then the sequence of measures $\bbP_n$ is tight, and we can take any weak limit of $\bbP_n$.
\end{proof}


\appendix
\renewcommand{\thesection}{\Alph{section}}
\numberwithin{equation}{section}

\section{Proof of Theorem \ref{RPF 0}}\label{Sec app A}

\subsection{Proof of Theorem \ref{RPF 0}(i)}
Let $\ve_0>0$ be such that $ \eta:=\rho+\ve_0<1$ and let 
 $a>0$ be a constant such that  $a\ve_0>K_0$, where $\rho$ and $K_0$ come from \eqref{LY Iter}. 
 Then it follows from \eqref{LY Iter} that for every $n\geq N$ and all $j\geq 0$ we have 
 $$
\cL_j^n\cC_{j,a}\subset\cC_{j+n,a\rho^{1+N-n}+K_0}\subset \cC_{j+n,\eta a}.
 $$
 Given $C>0$ consider the cone 
 $$
 \cK_j(C)=\{g\in B_j, \;\; g>0, \quad \text{esssup }g\leq C\text{ essinf }g\}
 $$
 where  the essential supremum and infimum are with respect to $m_j$. 
Denote by $\mathsf{d}_j$ the projective Hilbert metric associated with the cone $\cC_{j,a}$ (see \cite{Bir, Liver}).
 Then by \cite[(1.1)]{Buzzi} 
 the projective diameter of $\cK_j(C)\cap \cC_{j,\eta a}$ inside $\cC_{j,a}$ does not exceed a constant $R=R(a,\eta,V,C)$ which depends only on $a,\eta,C$ and 
 $\DS V:=\sup_j v_j(\textbf{1})$.  Next, combining (V3), (SC), \eqref{LY Iter} and the relation $(\cL_j^n)^*m_{j+n}=m_j$, we see that there are constants $M\geq N$ and $C$ (depending on $a$ but not on $j$) such that for every $j$ and $n\geq M$ we have 
 $$
 \cL_j^n\cC_{j,a}\subset \cK_{j+n}(C).
 $$
 We conclude that for every $j$ and $n\geq M$ the projective diameter of $\cL_j^n\cC_{j,a}$ 
 inside $\cC_{j+n,a}$ does not exceed $R$.

The next step of the proof is the following result.
\begin{lemma}\label{LemProps}
(a) There is a constant $A>0$ such that for every  $a$, $j$ and $h\in\cC_{j,a}$, 
$$
\|h\|\leq Am_j(h)
$$
\vskip0.1cm
(b)  If  $\DS a>V$ then there is  a constant $r_0>0$ such that  every real valued $g\in B_j$ can be written in the form $g=g_1-g_2$ where $g_i\in\cC_{j,a}$ and $\|g_1\|+\|g_2\|\leq r_0\|g\|$.
\end{lemma}

\begin{proof}
To prove (a) we note that if $h\in\cC_{j,a}$ then $\|h\|=m_j(h)+v_j(h)\leq (1+a)m_j(h)$ so we can take $A=1+a$.
\vskip0.1cm

In order to prove (b), given $g\in B_j$ we take $C_0=C_0(g)=\|g\|_\infty+\frac{1+a}{a-V}\|g\|_{BV}$. It is immediate to check that $ g+C_0\in\cC_{j,a}$. Since $C_0\leq C\| g\|_{BV}$ for some constant $C$ which does not depend on $j$ we have 
$\DS
g=(g+C_0)-C_0
$
and, with $\|\cdot\|=\|\cdot\|_{BV}$,
$$
 \|g+C_0\|+\|C_0\|\leq \|g\|+2\|C_0\|\leq \|g\|(1+2C_0+2C_0V)\leq r\|g\|
$$
for some constant $r=r(a,V,C)$. Thus we can take $g_1=(g+C_0)$ and $g_2=C_0$ (the constant function).
Note that $g_2\in\cC_{j,a}$ since $a>V$.
\end{proof}

Based on Lemma \ref{LemProps}, the proof of proof of Theorem \ref{RPF 0}(ii) is completed 
like in \cite[Chapters 4 and 5]{HK}. 
In fact, here the situation is much simpler, and  for readers' convenience we provide  the details.
First, by taking $f=1$ in \eqref{Duality} and letting $g$ vary we see that $(\cL_j^n)^*m_{j+n}=m_j$ for all $j$ and $n$.  Next, fix  some $a$ large enough and take two sequences of functions $g_n,f_n\in\cC_{n,a}$. Denote $G_{j,n}=\cL_{j}^n g_{j}$ and $F_{j,n}=\cL_{j}^n f_{j}$.  Then, since the projective diameter of the image of $\cL_k^s\cC_{a,k}$ inside $\cC_{a,k+s}$ does not exceed $R$ if $s\geq M$, we see  by \cite[Theorem 1.1]{Liver} that for every $j$ and $n\geq M$ we have 
$$
\mathsf{d}_j(F_{j,n},G_{j,n})\leq (c(R))^{[(n-1)/M]}
$$
where $c(R)=\tanh(R/4)\in(0,1)$. 

 For a positive function $g$, define the normalized iterates by 
$$
\bar\cL_{j}^n g=\frac{\cL_j^n g}{m_{j+n}(\cL_j^n g)}=\frac{\cL_j^n g}{m_{j}(g)}.
$$
Then by \cite[Theorem A.2.3]{HK} and Lemma \ref{LemProps}(1) we have 
\begin{equation}\label{exp temp}
\left\|\bar\cL_{j}^nf_{j} -\bar\cL_{j}^n g_{j}\right\|_{BV}\leq \frac 12 A (c(R))^{[(n-1)/M]}.
\end{equation}
Next, fix some  function $h_0\in \cC_{a,0}$ such that $m_0(h_0)=1$ and for $j>0$ define $h_j=\cL_0^j h_j$. Then $h_{j+1}=\cL_jh_j$ and $m_j(h_j)=1$ for all $j\geq0$. Therefore, 
\begin{equation}\label{h id}
\bar\cL_{j}^n h_{j}=h_{j+n} 
\end{equation}
for all $j\geq0$ and $n>0$. Moreover, $h_j\in \cC_{a,j}$ (due to the equivariance of the cones). 
By taking $g_{j}=h_{j}$ in \eqref{exp temp} and using \eqref{h id} we see that 
$$
\left\|\bar\cL_{j}^nf_{j} -h_{j+n}\right\|_{BV}\leq \frac 12 A (c(R))^{[(n-1)/M]}.
$$
Multiplying by $m_j(f_j)$ we get that, with $\del=(c(R))^{1/M}$, for all $f\in \cC_{a,j}$ we have 
$$
\|\cL_j^n(f)-m_j(f)h_{j+n}\|\leq m_j(f) B\del^n
$$ 
where $B$ is some constant. Using Lemma \ref{LemProps}(b) we obtain \eqref{Exp Conv Main}
 with $C=2r_0B$.
\qed

\subsection{Proof of Theorem \ref{RPF 0}(ii)}
To see why $(T_j)_*\mu_j=\mu_{j+1}$, take a measurable function $g:X_j\to\bbR$ and write
$$
\left[(T_j)_*\mu_j\right](g)=\mu_j(g\circ T_j)=m_j(h_j g\circ T_j)=m_{j+1}((\cL_j h_j)\cdot g)=
m_{j+1}(h_{j+1}g)=\mu_{j+1}(g).
$$
To prove  asymptotic uniqueness of $\mu_j$, let 
$\tilde \mu_j\!\!=\!\!g_jdm_j$ 
be another sequence of probability measures such that $(T_j)_*\tilde\mu_j=\tilde\mu_{j+1}$. 
We show that 
$\DS \lim_{j\to\infty}\|h_j\!-\!g_j\|_{1}\!\!=\!\!0$. We claim that 
\begin{equation}\label{g equ}
 \cL_j g_j=g_{j+1},\quad  m_{j+1}-\text{a.e.}   
\end{equation}   
Indeed given a function $f:X_{j+1}\to\bbR$ we have 
$$
m_{j+1}((\cL_j g_j)f)=m_j(g_j(f\circ T_j))=\tilde\mu_{j}(f\circ T_j)=\tilde\mu_{j+1}(f)=m_{j+1}(g_{j+1}f)
$$
proving \eqref{g equ}.
Iterating \eqref{g equ} we see that $g_n=\cL_0^n g_0$, $m_n$-a.e. Next, let $\ve>0$ and let $q_0:X_0\to\bbR$ be a BV function such that $\|q_0-g_0\|_{1}<\ve$. By normalizing $q_0$ if needed we can always assume that $m_0(q_0)=1$.
By Theorem \ref{RPF 0}(i) we see that 
$$
\|\cL_0^nq_0-h_n\|_{BV}\leq C\|q_0\|_{BV}\del^n\to 0, \text{ as }n\to\infty.
$$
On the other hand,
$$
\|\cL_0^nq_0-g_n\|_1=\|\cL_0^nq_0-\cL_0^ng_0\|_1\leq \|\cL_0^n(|q_0-g_0|)\|_1=\|g_0-d_0\|<\ve.
$$
Therefore, by taking $n$ large enough we get that $\|h_n-g_n\|_1<\ve$, and the proof of Theorem \ref{RPF 0}(ii) is complete.
   \qed

\section{Gibbs measures for SFT.} 
\label{ScSFT-Gibbs}

In this section we denote by
 $T_j:X_j\to X_{j+1}$, 
  a one sided topologically mixing sequential subshift of finite type and by 
 $\tilde T_j:\tilde X_j\to  \tilde X_{j+1}$, $j\in\mathbb{Z}$
  the corresponding two sided shift. 
  For a point $(x_j,x_{j+1},...)\in X_j$ and $m\in \bbN$ let
 $$
[x_j,x_{j+1},...,x_{j+m-1}]=\left\{(x'_{j+k})_{k\geq 0}\in X_j: x'_{j+k}=x_{j+k},\,\forall k<m\right\}
 $$
 be  the corresponding cylinder of length $m$. 
 Cylinders in $\tilde X_j$ are denoted similarly.

\begin{remark}
The reason we consider here negative indexes $j$ is to provide a complete theory which includes uniqueness of Gibbs measures (see Section \ref{g sec}) and relations with two sided shift via a non-stationary version of Sinai's lemma (see Lemma \ref{Sinai}). If we begin with a one sided shift then there are many Gibbs measures corresponding to a given sequence of potentials, each of which corresponds to an extension of the shift to negative indexes and an extension of the sequence of potentials to a two sided sequence. Considering negative indexes $j$ also allows us to uniquely define two sided shifts which have applications to non-autonomous hyperbolic systems (see Appendix \ref{App C}).
\end{remark}

\subsection{A sequential  Sinai's Lemma}
Let $\pi_j:\tilde X_j\to X_j$ be given by 
$$
\pi_j((x_{j+k})_{k\in\bbZ})=(x_{j+k})_{k\geq 0}.
$$

\begin{lemma}\label{Sinai}
 Fix some $\al\in(0,1]$ and let $\psi_j:\tilde X_j\to\bbR$ be uniformly H\"older continuous with exponent $\al$. Then there are uniformly H\"older continuous functions $u_j:\tilde X_j\to\bbR$ with exponent $\al/2$ and $\phi_j:X_j\to\bbR$  such that 
 $$
\psi_j=u_{j}-u_{j+1}\circ \sig_j+\phi_j\circ \pi_j.
 $$
 Moreover
  if $\|\psi_j\|_{\al}\to 0$ then $\|u_j\|_{\al/2}\to 0$. 
 \end{lemma}
\begin{proof}
The proof is a modification of the proof of \cite[Lemma 1.6]{Bow}.
For each $j$ and $t$ take a point $a^{(j,t)}=(a^{(j,t)}_{j+k})_k\in \tilde X_j$ such that $a_{j}^{(j,t)}=t$. For $y=(y_{j+k})_k\in \tilde X_j$ define
$$
r_j(y)=(y^*_{j+k})_{k\in\bbZ}\in \tilde X_j
$$
where $y^*_{j+k}=y_{j+k}$ if $k\geq 0$ and $y^*_{j+k}=a^{(j,y_{j})}_{j+k}$ if $k\leq 0$. Let
\begin{equation}
\label{DefU}
u_j(y)=\sum_{k=0}^\infty\left(\psi_{j+k}(\tilde T_j^k y)-\psi_{j+k}(\tilde T_j^k r_j(y))\right).
\end{equation}
To see that  the RHS of \eqref{DefU}  converges note that since $y$ and $r_j(y)$ have the same values at the coordinates indexed by $j\!+\!k$ for $k\!\!\geq\!\!0$ (i.e. with  indexes  to the right of $j$) we have
\begin{equation}\label{Pass 1}
    |\psi_{j+k}(\tilde T_j^k y)-\psi_{j+k}(\tilde T_j^k r_j(y))|\leq v_\al(\psi_{j+k})2^{-\al k}
\end{equation}
and so 
$\DS
\sup|u_j|\leq \sum_{k=0}^\infty v_\al(\psi_{j+k})2^{-\al k}.
$
Thus, $\DS \sup_j\sup_{\tX_j} |u_j|<\infty$. Moreover, 
if $v_{\al}(\psi_j)\to 0$ then $\DS \sup_{\tX_j} |u_j|\to 0$.

Next, we claim that 
the function $\psi_j-u_j+u_{j+1}\circ\sig_j$ depends only on the coordinates indexes by $j+k$ for $k\geq 0$ (so it has the form $\phi_j\circ \pi_j$).
Indeed,
$$
u_j-u_{j+1}\circ\tilde T_j=\psi_{j}+\sum_{k=0}^\infty\left(\tilde T_{j+k}\circ \tilde T_j^k\circ r_j-\psi_{j+k+1}\circ \sig_{j+1}^k\circ r_{j+1}\circ \sig_j\right).
$$
Note the the sum on the above right hand side depends only on the coordinates indexed by $j+k$ for $k\geq 0$, and the claim follows.

In view of the above, in order to complete the proof of the lemma it is enough to obtain appropriate bounds on the H\"older constants of the functions $u_j$ (corresponding to the exponent $\al/2$). 
For that end, let $y$ and $y'$  in $\tilde X_j$ be such that $y_{j+k}=y'_{j+k}$ for every $|k|\leq n$ for some $n>0$. 
Using \eqref{Pass 1} we have 
$$
|u_j(y)-u_j(y')|\leq \sum_{k=0}^{[n/2]}\left|\psi_{j+k}(\tilde T_j^k y)-\psi_{j+k}(\tilde T_j^k y')\right|+
\sum_{k=0}^{[n/2]}\left|\psi_{j+k}(\tilde T_j^k r_j(y))-\psi_{j+k}(\tilde T_j^k r_j(y'))\right|
$$
$$
+2\sum_{k>[n/2]}v_\al(\psi_{j+k})2^{-\al k}:=I_1+I_2+I_3.
$$
To show that $u_j$ is H\"older continuous with exponent $\al/2$ (uniformly in $j$) we use that 
$$
\left|\psi_{j+k}(\tilde T_j^k y)-\psi_{j+k}(\tilde T_j^k y')\right|\leq \sup_sv_\al(\psi_s)2^{-(n-k)\al}
$$
and similarly with $r_j(y)$ and $r_j(y')$ instead of $y$ and $y'$, respectively.  So
$I_1+I_2\leq C2^{-n\al/2}$.
Moreover, we note that $I_3\leq C2^{-n\al/2}$. 

 Next, if $\|\psi_j\|_{\al}\to 0$ then 
 $\DS
I_1+I_2\leq\sum_{k=0}^{[n/2]}v_\al(\psi_{j+k})2^{-(n-k)\al}\leq \ve_j2^{-n\al/2}
 $
 with $\ve_j\to 0$. 
 Similarly, $I_3\leq \ve_j 2^{n\al/2}$ for $\ve_j$ with the same properties.
 \end{proof}

\subsection{Sequential Gibbs measures for one sided shifts}\label{g sec}

\begin{definition}
Let $\phi_j:X_j\to\bbR$ be a sequences of functions such that $\DS \sup_j \|\phi_j\|_\al<\infty$ for some $\al\in(0,1]$.  We say that a sequence of probability measures $\mu_j$ on $X_j$ is a {\em sequential Gibbs family for $(\phi_j)$} if: 

\vskip0.2cm
(i) For all $j$ we have $(T_j)_*\mu_j=\mu_{j+1}$
\vskip0.2cm

(ii) There is a constant $C>1$ and a sequence of positive numbers $(\la_j)$ such that for all $j$, every point $(x_{j+k})_k$ in $X_j$ and every $r>0$ we have 
$$
C^{-1}e^{S_{j,r}\phi(x)}/\la_{j,r}\leq \gamma_j([x_j,...,x_{j+r-1}])\leq Ce^{S_{j,r}\phi(x)}/\la_{j,r}
$$
where 
$\DS
S_{j,r}\phi(x)=\sum_{s=0}^{r-1}\phi_{j+s}(T_{j}^s x)
$
and 
$\DS
\la_{j,r}=\prod_{k=j}^{j+r-1}\la_k.
$
\end{definition}

We say that two sequences $(\al_j)$ and $(\be_j)$
  of positive numbers are \textit{equivalent} if there is a sequence $(\zeta_j)$ of positive numbers which is bounded and bounded away from $0$ such that for all $j$ we have $\al_j=\zeta_j\beta_j/\zeta_{j+1}$.

  We need the following simple result.
\begin{lemma}\label{EquivLemma}
  Two positive sequences $(\al_j)$ and $(\beta_j)$ are equivalent if and only if there is a constant $C>0$ such that for all $j$ and $n$ we have 
  $\DS
  C\leq \al_{j,n}/\beta_{j,n}\leq C^{-1}.
  $
  \begin{proof}
      It is clear that there exists such a constant $C$ if  $(\al_j)$ and $(\be_j)$ are equivalent. On other hand, suppose that such a constant exists. Then, for $j\geq 0$ define 
      $
\zeta_j=\beta_{0,j}/\al_{0,j}.
      $
      Clearly,
      $\DS
\al_{j}/\beta_{j}=\zeta_{j}/\zeta_{j+1}.
      $
      For $j<0$ we define similarly 
      $
\zeta_j=\al_{j+1,|j|}/\beta_{j+1,|j|}.
      $
  \end{proof}
  \end{lemma}

\begin{theorem}
   For every sequence of functions $\phi_j:X_j\to\bbR$ such that $\DS \sup_j\|\phi_j\|_\al<\infty$ for some $\al\in(0,1]$ there exists  unique Gibbs measures $\mu_j$. 
   Moreover
     the sequence $(\la_j)$ is unique up to equivalence. 
     \end{theorem}

 We will see below that the measures $\mu_j$ 
can be  expressed in terms of the strong limits of the normalized operators $\cL_j^n$ associated with the sequence $(\phi_j)$ defined in \eqref{Norm TO}.  The rate of convergence is exponential, see \eqref{ExpConvSTF} below.

\begin{proof}
    Existence of $\mu_j$ was proven in \cite{Nonlin}. Let us recall the main arguments. First, we define the operator $L_j$ 
   which maps a function $g:X_j\to \bbR$ to a function $L_j g:X_{j+1}\to\bbR$ given  by
    $$
L_j g(x)=\sum_{T_j y=x}e^{\phi_j(y)}g(y).
    $$
Then, in \cite{Nonlin} it was shown that there is a positive sequence $(\la_j)$ which is bounded and bounded away from the origin, a sequence of positive functions $h_j$ with uniformly bounded H\"older norms (corresponding to the same exponent $\al$ of $\phi_j$), which are also uniformly bounded below by a positive constant, and a sequence of probability measures $\nu_j$ on $X_j$ such that $\nu_j(h_j)\!\!=\!\!1$ and  for all $j,n$ and a H\"older continuous function $g$ on $X_j$,
\begin{equation}\label{ExpConvSTF0}
    \|(\la_{j,n})^{-1}L_j^n g-\nu_j(g)h_{j+n}\|_{\al}\leq C_0\|g\|_\al \del^n
\end{equation}
for some constants $C>0$ and $\del\in(0,1)$ which do not depend on $j,n$ and $g$. Here
$$
L_{j}^n=L_{j+n-1}\circ\cdots\circ L_{j+1}\circ L_j.
$$
Then the Gibbs measures constructed in \cite{Nonlin} are given by $\mu_j=h_j d\nu_j$. Moreover, the operator $\cL_j$ given by 
\begin{equation}\label{Norm TO}
\cL_jg(x)=\frac{L_j(gh_j)}{\la_jh_{j+1}}
\end{equation}
is the dual of the Koopman operator corresponding to $T_j$ with respect to $\mu_j$ and $\mu_{j+1}$ and (because of \eqref{ExpConvSTF0}),
\begin{equation}\label{ExpConvSTF}
    \|\cL_j^n g-\mu_j(g)\textbf{1} \|_{\al}\leq C\|g\|_\al \del^n
\end{equation}
for some $C>0$. In the derivation of \eqref{ExpConvSTF} we used that $\DS \sup_j\|1/h_j\|_\al<\infty$.

Now, let us suppose that there is another Gibbs measure $\tilde\mu_j$ associated with a sequence $\tilde \la_j$. Namely, $(T_j)_*\tilde\mu_j=\tilde \mu_{j+1}$ and  for every point $x\in X_j$,
$$
(\tilde C)^{-1}e^{S_{j,r}\phi(x)}/\tilde \la_{j,r}\leq \tilde\mu_j([x_{j},x_{j+1},...,x_{j+r-1}]])\leq \tilde Ce^{S_{j,r}\phi(x)}/\tilde\la_{j,r}
$$
for some constant $\tilde C>0$. Let us first show that the sequences $(\la_j)$ and $(\tilde\la_j)$ are equivalent. Indeed, for all $j$ and $n$,
 by the Gibbs property of both $\mu_j$ and $\tilde\mu_j$,
$$
C^{-1}\sum_{[y_j,...,y_{j+n-1}]}e^{S_{j,n}\phi(x_y)}/\la_{j,n}\leq \sum_{[y_j,...,y_{j+n-1}]}\mu_j([y_j,...,y_{j+n-1}])\leq C\sum_{[y_j,...,y_{j+n-1}]}e^{S_{j,n}\phi (x_y)}/\la_{j,n}
$$
and 
$$
(\tilde C)^{-1}\sum_{[y_j,...,y_{j+n-1}]}e^{S_{j,n}\phi (x_y)}/\tilde \la_{j,n}\leq \sum_{[y_j,...,y_{j+n-1}]}\tilde \mu_j([y_j,...,y_{j+n-1}])\leq \tilde C\sum_{[y_j,...,y_{j+n-1}]}e^{S_{j,n}\phi (x_y)}/\tilde\la_{j,n}
$$
where the point $x_y$ is an arbitrary point inside the cylinder $[y_j,...,y_{j+n-1}]$.
On the other hand,
$$
1=\sum_{[y_j,...,y_{j+n-1}]}\mu_j([y_j,...,y_{j+n-1}])=\sum_{[y_j,...,y_{j+n-1}]}\tilde\mu_j([y_j,...,y_{j+n-1}]).
$$
Hence,
$$
(C\tilde C)^{-1}\leq \frac{\la_{j,n}}{\tilde \la_{j,n}}\leq C\tilde C
$$
and the equivalence of the sequences $(\la_j)$ and $(\tilde \la_j)$ follows from Lemma \ref{EquivLemma}.

Next, let us show that $\tilde\mu_j=\mu_j$ for all $j$. 
We first need the following result. Let $H_{j,\al}$ denote the space of H\"older continuous functions on $X_j$ with the H\"older  exponent $\al$, equipped with the usual $\al$-H\"older norm $\|\cdot\|_\al$. Let $H_{j,\al}^*$ denote its dual, and let us denote by $\|\cdot\|_\al$ the (operator) norm on the dual, as well. Let  $\cL_j^*:H_{j+1,\al}^*\to H_{j,\al}^*$  be the dual operators of $\cL_j$. Then it follows from \eqref{ExpConvSTF} that for every $j,n$ and $\ka_{j+n}\in H_{j+n,\al}^*$, 
\begin{equation}\label{DualExpConv}
 \|(\cL_j^n)^*\ka_{j+n}-\ka_{j+n}(\textbf{1}_{j+n})\mu_{j}\|_\al\leq C\|\ka_{j+n}\|_\al\del^n   
\end{equation}
where $\textbf{1}_k$ is the function on $X_k$ which takes constant value $1$.
In particular, if we take $\ka_{j+n}=\tilde\mu_{j+n}$ then, since $\tilde\mu_{j+n}\textbf{1}_{j+n}=1$, we see that 
\begin{equation}\label{Converge}
(\cL_j^n)^*\tilde\mu_{j+n}\to \mu_j\,\text{ as }\, n\to\infty.
\end{equation}

Next,  we claim that 
\begin{equation}\label{NWCT1}
(\cL_j^n)^*\tilde \mu_{j+n}=\tilde \mu_j.
\end{equation}
Once the claim is proven, combining it with \eqref{Converge} we obtain $\tilde \mu_j=\mu_j$, as required. 

Now, because of the Gibbs properties of both $\mu_j$ and $\tilde\mu_j$ and since $(\la_{j})$ and $(\tilde\la_j)$ are equivalent we see that
$\tilde \mu_j\ll \mu_j$. Let $ p_j=\frac{d\tilde\mu_j}{d\mu_j}$
 denote the
corresponding Radon Nikodym derivative.
In order to prove \eqref{NWCT1},  we will show that for every $j$,
\begin{equation}\label{p dens}
   p_{j+1}\circ T_j=p_j,\quad\mu_j-\text{a.s.}
\end{equation}
 \eqref{p dens} implies that $p_{j+n}\circ T_j^n=p_j$.
Therefore, for every bounded measurable function $g:X_j\to\bbR$,
$$
(\cL_j^n)^*\tilde\mu_{j+n}(g)=
\tilde\mu_{j+n}(\cL_{j}^n g)=\mu_{j+n}(p_{j+n}\cL_{j}^n g)=\mu_j(g\cdot (p_{j+n}\circ T_j^n))=\mu_j(g\cdot p_j)=\tilde\mu_j(g)
$$
and \eqref{NWCT1} follows (note that in the third equality above we have used that $\cL_j^n$ is the transfer operator of $T_j^n$ with respect to $\mu_j$ and $\mu_{j+n}$).

To complete the proof of the theorem we need to prove \eqref{p dens}.  
 By identifying both measures $\mu_j$ and $\tilde \mu_j$ as measure on the two sided shift $\tilde X_j$ (see next section), we can assume that $\mu_j$ and $\tilde \mu_j$ are mapped by $\tilde T_j$ to $\mu_{j+1}$ and $\tilde\mu_{j+1}$, respectively. In what follows we will abuse the notation and denote the lifted measures by $\mu_j$ and $\tilde \mu_j$. The identification of the function $p_j$ to $\tilde X_j$ will also be denoted by $p_j$, namely, we  write $p_j(...,x_{j-1},x_{j},x_{j+1},...)=p_j(x_{j},x_{j+1},...)$.
Now, since
$\DS
(\tilde T_j)_*\mu_j=\mu_{j+1}$
and
$\DS (\tilde T_j)_*\tilde \mu_j=\tilde\mu_{j+1}
$
for every bounded measurable function $f:\tilde X_{j+1}\to\bbR$ we have 
$$
\mu_j(p_j(f\circ \tilde T_j))=\tilde\mu_j(f\circ \tilde T_j)=\tilde\mu_{j+1}(f)=\mu_{j+1}(p_{j+1}f)=
\mu_j((p_{j+1}\circ \tilde T_j)\cdot (f\circ \tilde T_{j})).
$$
Since $T_j$ is invertible we can replace $f\circ T_j$ with a general bounded measurable function $g$ on $X_j$ and \eqref{p dens} follows.
\end{proof}

\subsection{Gibbs measure for two sided shifts}

\begin{definition}
Let $\psi_j:\tilde X_j\to\bbR$ be a sequences of functions such that 
$\DS \sup_j\|\psi_j\|_\al<\infty$ for some $\al\in(0,1]$.  We say that a sequence of probability measures $\gamma_j$ on $\tilde X_j$ is a Gibbs family for $(\psi_j)$ if: 

\vskip0.2cm
(i) For all $j$ we have $(\tilde T_j)_*\gamma_j=\gamma_{j+1}$.
\vskip0.1cm

(ii) There is a constant $C>1$ and a sequence of positive numbers $(\la_j)$ so that for every point $(y_{j+k})_k$ in $\tilde X_j$ we have 
$$
C^{-1}e^{S_{j,r}\psi(y)}\la_{j,r}\leq \gamma_j([y_{j},y_{j+1},...,y_{j+r-1}]])\leq Ce^{S_{j,r}\psi(y)}\la_{j,r}
$$
where 
$\DS
S_{j,r}\psi(y)=\sum_{s=0}^{r-1}\phi_{j+s}(\tilde T_{j}^s y)
$
and 
$
\DS \la_{j,r}=\prod_{k=j}^{j+r-1}\la_k.
$
\end{definition}

\begin{proposition}\label{Gibbs two sided}
    Every uniformly H\"older continuous sequence $\psi_j$ (with respect to some exponent $\al\in(0,1]$) admits a unique sequence of Gibbs measures $\gamma_j$. 
\end{proposition}
\begin{proof}
 Let $\phi_j$ be the functions obtained in Lemma \ref{Sinai}. Let $\mu_j$ be the unique sequential Gibbs measure on the one sided shift space $X_j$ corresponding to the functions $(\phi_j)$. 
    Since 
$(T_{j-k})_*\mu_{j-k}=\mu_j$ using Kolmogorov's extension theorem we can extend $\mu_j$ to the space of two sided sequences $\tilde X_j$. Let us denote this measure by $\gamma_j$. It is clear that $(\tilde T_j)_*\gamma_j=\gamma_{j+1}$. 
To prove the second condition in the definition of a sequential Gibbs measure we notice that for every point $x\in \tilde X_j$,
$$
\gamma_j([x_{j},...,x_{j+r-1}])=\mu_j([x_{j},...,x_{j+r-1}])\asymp \la_{j,n}e^{S_{j,r}\phi\circ \pi_j(x)}
$$
where $(\la_j)$ is the sequence associated with the Gibbs measure $\mu_j$.
Now, by Lemma \ref{Sinai} we have that 
$$
e^{S_{j,r}\phi\circ \pi_j(x)}\asymp e^{S_{j,r}\psi(x)}
$$
and hence $\gamma_j$ is a sequential Gibbs measure corresponding to the functions $\psi_j$, and $(\la_j)$ is the corresponding associated sequence.

To prove uniqueness of $\gamma_j$ constructed above, we note that if $\tilde \gamma_j$ is another Gibbs measure then the measure $\tilde \mu_j$ defined by the restriction of $\gamma_j$ to the $\sig$-algebra generated by the coordinates with indexes $j+k, k\geq 0$ projects\footnote{By ``projects" we mean that if the restriction is denoted by $\gamma_j^*$ then $(\pi_j)_*\gamma_j^*=\mu_j$.} to a Gibbs measure on $X_j$. Hence $\tilde \mu_j=\mu_j$, where $\mu_j$ is the unique sequential Gibbs measure corresponding to the functions $\phi_j$. On the other hand, since $(\tilde T_j)_*\tilde \gamma_j=\tilde\gamma_{j+1}$, for every $r>0$ the restriction of $\tilde\gamma_j$ to the $\sig$-algebra generated by the coordinates indexed by $j+k$ for $k\geq -r$ coincides with the restriction of $\tilde\gamma_{j-r}$ to the  $\sig$-algebra generated by the coordinates indexed by $j-r+k$ for $k\geq 0$, which, as explained above, projects on $X_{j-r}$ to  $\mu_{j-r}$. Hence, for every $r$ the measures $\gamma_j$ and $\tilde\gamma_{j}$ agree on the $\sig$-algebra generated by the coordinates indexed by $j+k$ for $k\geq -r$ (as both coincide with $\mu_{j-r}$).  By taking $r\to-\infty$ we conclude that 
$
\tilde \gamma_j=\gamma_j.
$
\end{proof}


\section{Small perturbations of hyperbolic maps}\label{App C}
\subsection{Hyperbolic sets.}
Let $M$ be a compact $C^2$ Riemannian
manifold  equipped with its Borel $\sig$-algebra $\cB$. Let us also denote by $\mathsf{d}(\cdot,\cdot)$ the induced metric.
Let $T:M\to M$ be a $C^2$ diffeomorphism.

\begin{definition}
A compact $T$ invariant subset $\Lambda\subset M$ is called a {\em hyperbolic set}  for $T$   if  there exists an open set $V$ with compact closure, constants $\la\in(0,1)$ and $\al_0,A_0,B_0>0$  and subbudnles $\Gamma^s$ and $\Gamma^u$ of the tangent bundle $T\Lambda$ such that:
\vskip0.1cm
(i) The set $\{x\in M: \text{dist}(x,\Lambda)<\al_0\}$ is contained in a open subset $U$ of $V$ such that $TU\subset V$ and $T|_{U}$ is a diffeomorphism with $\DS \sup_{x\in U}\max(\|D_xT\|, \|D_xT^{-1}\|)\leq A_0$;
\vskip0.1cm
(ii) $T\Lambda=\Gamma^s\oplus\Gamma^u$, $DT(\Gamma^s)=\Gamma^s$, $DT(\Gamma^u)=\Gamma^u$ and the minimal angle between $\Gamma^s$ and $\Gamma^u$ is bounded below by $\al$;
\vskip0.1cm
(iii) For all $n\in\bbN$ we have 
$$
\|D_xT^n v\|\leq B_0\la^n\|v\|\,\,\,\forall \,v\in \Gamma_x^s\quad\text{and}
\quad \|D_xT^{-n} v\|\leq B_0\la^{n}\|v\|\,\,\,\, \forall \, v\in \Gamma_x^u.
$$
\end{definition}

\begin{definition}

A hyperbolic set is called 

(i) {\em locally maximal} if the set $U$ above could be chosen 
so that $\DS \Lambda=\bigcap_{n\in \mathbb{Z}} T^n U$ (that is, $\Lambda$ is the largest hyperbolic set contained in $U$);

(ii) hyperbolic attractor, if in addition, $U$ could be chosen so that $TU\subset U$
 (in the case when $M=\Lambda$ $T$ is called {\em Anosov}). 

We say that $\Lambda$ is the {\em basic hyperbolic set} if it is infinite locally maximal hyperbolic 
set such that $T|_\Lambda$ is topologically transitive.
\end{definition}

\vskip-2mm
 Henceforth, we assume that  $\Lambda$ is topologically mixing\footnote
{Topological mixing assumption can be made without a loss
of generality. Indeed (see e.g. \cite[Chapter 8]{Sh}) 
an arbitrary basic set $\Lambda$ can be decomposed
as $\DS \Lambda=\bigcup_{j=1}^p \Lambda_j$ so that 
$T\Lambda_j=\Lambda_{j+1\;\;mod\;\;p}$ where $\Lambda_j$ are topologically mixing
basic hyperbolic sets for $T^p.$ Then we could apply the results discussed below
to $(T^p, \Lambda_j).$}
 basic 
hyperbolic set.
\\

A powerful tool for studying hyperbolic maps is given by a symbolic representations. 
Namely,  every topologically mixing basic set $\Lambda$ admits a Markov partition (see \cite[Chapter 10]{Sh})
 which gives raise to a 
semiconjugacy $\pi:\Sigma\to M$ where $\Sigma$ is a
topologically mixing subshift of a finite type.

\subsection{Structural stability}
Now, consider a sequence of maps $\cT=(T_j:M\to M)_{j\in\bbZ}$. 
 Denote by $\mathsf{d}_1(f,g)$ the $C^1$ distance between $f$ and $g$.
We have the following result.

\begin{theorem}
\label{ThStability}
If $\DS \del_{1}(\cT):=\sup_j\mathsf{d}_{1}(T,T_j)$  is small enough then there 
is a sequence  of sets $\Lambda_j\subset M$ and homeomorphisms  $h_j:\Lambda\to\Lambda_j$ (that we think of as a ``sequential conjugacy")  
such that 
$h_j$ and $h_j^{-1}$ are uniformly 
H\"older continuous,
  \begin{equation}\label{Cong}
 T_j\Lambda_j=\Lambda_{j+1}\, \text{ and }\, T_j\circ h_j=h_{j+1}\circ T.
 \end{equation}
Moreover $\DS \sup_j\|h_j-\text{Id}\|_{C^{0}}\to 0$ as $\del_1(\cT)\to 0.$ 
\footnote{Note that in Theorem \ref{ThStability} we can also consider one sided sequences $(T_j)_{j\geq0}$ because they can be extended to two sided ones. The reason we consider two sided sequences and not one sided is  because the definition of hyperbolicity requires considering negative times to 
 define the unstable subspaces.}
\\

The sets  $(\Lambda_j)$ are sequentially hyperbolic for the sequence $\cT$ in the  following sense.
They are compact, satisfy $T_j\Lambda_j=\Lambda_{j+1}$ and there exist constants $\la'\in(0,1)$ and $\al_1, A_1,B_1>0$  and sequences of subbudnles 
$\Gamma_j^s\!\!=\!\!\{\Gamma_{j,x}^s: x\in \Lambda_j\}$ and 
$\Gamma_j^u\!\!=\!\!\{\Gamma_{j,x}^u: x\in \Lambda_j\}$ of the tangent bundle $T\Lambda_j$
 such that, for each $j$:
\vskip0.2cm
(i) The set $\{x\in M: \mathsf{d}(x,\Lambda_j)<\al_1\}$ is contained in an open  subset $U_j$ of $V$ such that $T_jU_j\subset V$ and $T_j|_{U_j}$ is a diffeomorphism satisfying
$$
\sup_j\sup_{x\in U_j}\max(\|D_xT_j\|, \|D_xT_j^{-1}\|)\leq A_1;
$$
\vskip0.2cm
(ii) $T\Lambda_j=\Gamma_j^s\oplus\Gamma_j^u$, $DT_j(\Gamma_j^s)=\Gamma_{j+1}^s$, $DT_j(\Gamma_j^u)=\Gamma_{j+1}^u$ and the minimal angle between $\Gamma_j^s$ and $\Gamma_j^u$ is bounded below by $\al_1$;
\vskip0.2cm
(iii) For every $n\in\bbN$ and all $j$ we have 
$\DS
\|D_xT_j^n v\|\leq B_1\cdot (\la')^n\|v\|\,\,\,\text{for every}\,\,v\in \Gamma_{j,x}^s
$

and 
$\DS
\|D_xT_j^{-n} v\|\leq B_1\cdot (\la')^{n}\|v\|\,\,\,\text{for every }\,\,v\in \Gamma_{j,x}^u
$
where $T_{j}^{-n}=(T_{j-n}^n)^{-1}$. 
\vskip0.2cm
(iv) $T_j U_j\subset U_{j+1}$ and 
 $\DS  \bigcap_{n=0}^\infty T_{j-n}^n U_{j-n}=\Lambda_j$.
\end{theorem}

\noindent
 Theorem \ref{ThStability} is proven in 
\cite[Example 4.2.3]{KiferLiu}, \cite[Example 2.5]{GunKif} and \cite[Theorem~1.1]{Liu JSP 1998}) except  the H\"older continuity of $h_j$ and $h_j^{-1}$ was not discussed there.
The proof of H\"older continuity is quite standard but for completeness it is  included in
\S \ref{h sec}. 
We note that in \cite{KiferLiu, GunKif, Liu JSP 1998}    the smallness of $\|T_j-T\|_{C^1}$ was not  uniform in $j$, which led to non-uniform in $j$ hyperbolicity. However, when $\del_1(\cT)$ is 
 uniformly small then the arguments in the proofs yield the uniform hyperbolicity. 

\subsection{Limit Theorems.}
 Let $\pi_j=h_j\circ \pi$. Then $\pi_j$ provides a semiconjugacy between
the sequence $\cT$ and the subshift $\Sigma$ describing the symbolic dynamics
of $T.$ Given a sequence of H\"older functions $\phi_j$ on $\Lambda_j$
 let $\psi_j=\phi_j\circ \pi_j.$
Note that $\psi_j$ are H\"older continuous due to H\"older continuity of $\pi_j$.
Let $\nu_j$ be the Gibbs measures for $\{\psi_j\}$ which exist due to the results
of Appendix \ref{ScSFT-Gibbs}.
Define measures $\mu_j$ on $\Lambda_j$ by
$\mu_j(A)=\nu_j(\pi_j^{-1} A).$
Given a sequence of Holder functions $f_j$ on $\Lambda_j$
let $g_j=f_j\circ \pi_j.$ Thus $S_n f=S_n g\circ \pi_0^{-1}$
and so
$(S_n f)(x)$ when $x$ is distributed according to $\mu_0$ has the the distribution as 
$(S_n g)(\omega)$ when $\omega$ is distributed according to $\nu_0.$ 
We thus obtain 

\begin{corollary}
Theorems \ref{BE}--\ref{ASIP} are valid for $(S_n f)(x)$ where $x$ is distributed according to 
$\mu_0.$
\end{corollary}

We also have the following result (see \cite[Theorem 4.3]{GunKif}).
\begin{theorem}[Sequential SRB measures]\label{SRBthm}
Suppose that $\Lambda$ is a hyperbolic attractor. Then
there is a sequence of probability measures $\mu_j$ on $\Lambda$ such that $(T_j)_*\mu_j=\mu_{j+1}$ and
$$
\mu_j=\lim_{n\to\infty}(T_{j-n}^n)_*(\rho_{j-n} d\text{Vol})
$$
for every   uniformly bounded sequence of probability densities $\rho_n$ on $\Gamma_{n}$.  The measures $(\mu_j),$ $j\in\bbZ$ are the unique family of equivariant measures such that the conditional measures of  $\mu_j$ on the unstable manifolds (see \S \ref{h sec}) at time $j$ are absolutely continuous with
respect to the Riemannian volume on these submanifolds.

 Moreover $\mu_j$ can be obtained by the construction described above corresponding to the sequence of 
 functions 
 $\phi_j=-\ln J(T_j|_{\Gamma_j^u})$, where $J(\cdot)$ stands for the Jacobian matrix. 

Therefore $\DS
W_n=\sum_{j=0}^{n-1}f_j(T_0^j x)$ satisfies
all the limit theorems in Section \ref{Sec2} 
 when $x$ is distributed according to a measure having a H\"older density with respect
 to $\mu_0.$
\end{theorem}

\subsection{H\"older continuity of the conjugacies}
\label{h sec}
To prove the H\"older continuity of $h_j$ and $h_j^{-1}$ we need some background.

\subsubsection*{Local stable and unstable manifolds}
For $\ve$ small enough, and $x\in\Lambda_j$ define $W_j^{s}(x,\ve)$ to be the set of all points $y\in \Lambda_j$ such that 
$\mathsf{d}(T_j^nx,T_j^n y)\leq \ve$ for all $n$ and $\mathsf{d}(T_j^nx,T_j^n y)\to 0$. Similarly, we define  $W_j^{u}(x,\ve)$ to be the set of all points $y\in \Lambda_j$ such that 
$$
\mathsf{d}((T_{j-n}^n)^{-1}x,(T_{j-n}^n)^{-1}y)\leq \ve
$$ 
for all $n$ and $\mathsf{d}((T_{j-n}^n)^{-1}x,(T_{j-n}^n)^{-1}y)\to 0$.  Then  (see \cite{GunKif, KiferLiu} and \cite{MR}) $W_j^{s}(x,\ve)$ and $W_j^u(x,\ve)$ are manifolds and  tangent space of $W_j^{s}(x,\ve)$ at $x$ is $\Gamma_{x,j}^s$ while the tangent space of $W_j^{u}(x,\ve)$ at $x$ is $\Gamma_{x,j}^u$. Moreover, there are constants $C>0$ and $\del\in(0,1)$ such that for every $j$,
\begin{equation}\label{Stab}
\mathsf{d}(T_{j}^n x, T_j^n y)\leq C\del^n\text{ for all }y\in W_j^s(x,\ve)
\end{equation}
and 
\begin{equation}\label{Uns}
\mathsf{d}((T_{j-n}^{n})^{-1} x, (T_{j-n}^n)^{-1} y)\leq C\del^n\text{ for all }y\in W_j^u(x,\ve)
\end{equation}

\begin{remark}\label{Rem Man}
Like in the autonomous case,  we have the following (see \cite[Proposition 3.5]{MR}).  
Define $\overline{W}_j^s(x,\beta)$ to be the set of all points $y\!\!\in\!\!\Lambda_j$ such that $\mathsf{d}(T_j^nx, T_j^n y)\!\!\leq \!\!
\beta$ for all $n\geq 0$. Similarly, let  $\overline{W}_j^u(x,\beta)$  be the set of all points $y\in \Lambda_j$ such that  $\mathsf{d}((T_{j-n}^n)^{-1}x,(T_{j-n}^n)^{-1}y)\leq \beta$  for all $n\geq0$. Then for every $\be<\frac 12\ve$,
$$
\overline{W}_j^{s}(x,\beta)\subset W_j^s(x,\ve)
\quad
\text{and} 
\quad
\overline{W}_j^u(x,\beta)\subset W_j^u(x,\ve).
$$
\vskip-2mm
This essentially means that, like in the autonomous case, up to replacing $\ve$ with $\frac12\ve$, 
the local stable/unstable manifolds can be defined using only the condition about the $\ve$-closeness of the forward/backward orbits.
\end{remark}
Let us also denote by $W^s(x,\ve)$ and $W^u(x,\ve)$ the local stable and unstable manifolds of $x$ with respect to $T$ (then all the above properties hold true).

\begin{proof}[Proof of H\"older continuity]
The proof is a minor modification of the proof of \cite[Proposition 19.1.2]{HassKat}, but for readers' convenience we provide the details.
We only prove that $h_j$ is H\"older continuous, the proof that $h_j^{-1}$ is  H\"older continuous is analogous, see below.
Fix some $\ve$ small enough (in a way that will be determined later). 
We say that $x,y\in\Lambda$ are $s$-equivalent if $y\in W^s(x,\ve)$.  Similarly, we say that they are $u$-equivalent if $y\in W^u(x,\ve)$.
Then, using local coordinates and that the angles between the stable and unstable directions are uniformly bounded below we see that if $\ve$ is small enough then there is a constant $K_1$ such that
\begin{equation}  
\label{LPMetric}
\mathsf{d}(x,y)^2+\mathsf{d}(y,z)^2\leq K_1\mathsf{d}(x,z)^2
\end{equation}
for every $x,y,z\in\Lambda$  such that $x$ is $s$-equivalent to $y$ and $y$ is $u$-equivalent to $z$, and $\mathsf{d}(x,z)\leq \ve.$ 
 In view of \eqref{LPMetric}
in order to  show that $h_j$ is H\"older continuous
it suffices to show that restrictions of $h_j$ to both unstable and stable manifolds are
uniformly H\"older (see e.g. \cite[Proposition 19.1.1]{HassKat}).
We will consider $h_j|W^u(x,\ve)$, the result for $h_j|W^s(x,\ve)$ follows by replacing $T$ 
by $T^{-1}.$

Let us prove that $h_j|W^u(x,\ve)$ are uniformly H\"older continuous. Fix some $\ve_0>0$.  By uniform continuity of $h_j$ proven in \cite{Liu JSP 1998},
there exists $0<\del_0<\ve$ such that for every $j$ and every $x,y\in\Lambda$ with $\mathsf{d}(x,y)<\del_0$ we have $\mathsf{d}(h_j(x),h_j(y))< \ve_0$.
Let $x,y\in\Lambda$ be such that $y\in W^u(x,\ve)$. Denote $\rho=\mathsf{d}(x,y)$ and let  $K>1$ be a Lipschitz constant for $T$. Assuming that $\rho \leq K^{-2}\del_0$  there is a unique natural number $n$ such that $K^n\rho<\del_0\leq K^{n+1}\rho$. Then $\mathsf{d}(T^k x, T^k y)\leq K^n\rho<\del_0$ for all $k\leq n$ and so 
$\DS
\mathsf{d}(h_{j+n}T^{n}x,h_{j+n}T^{n}{y})<\ve_0.
$
Since $y\in W^u(x,\ve)$ and $\mathsf{d}(T^k x, T^k y)\!\!<\!\!\del_0\!\!<\!\!\ve$ for $0\leq k\leq n$ we have
 $T^ny\in W^u(x,\ve)$. Using  Remark \ref{Rem Man} and the equicontinuity of  $(h_k)$,   we see that $h_{j+n}T^ny\in W_{j+n}^u(h_{j+n}T^nx,\ve')$ where 
 $\ve'\!\!=\!\!\ve'(\ve)\!\to\!\! 0$ as $\ve\!\to\!\! 0$ (we can take $\DS \ve'\!\!=\!\!\sup_k(\sup\{\mathsf{d}(h_k(a),h_k(b)): \mathsf{d}(a,b)\leq \ve\}$). 
 Thus, by fixing $\ve$ small enough
and using \eqref{Uns} (with the above $\ve'$ instead of $\ve$) we see that there are constants $C>0$ and $\del\in(0,1)$ such that
$$
\mathsf{d}(h_j(x), h_j(y))=\mathsf{d}((T_{j}^n)^{-1} h_{j+n}T^{n}x, (T_{j}^n)^{-1} h_{j+n}T^{n}y)\leq C\del^n\ve_0
$$
$$
=C\ve_0K^{-cn}=C\ve_0K^c (K^{-n-1})^c\leq (C\ve_0 K^c\del_0^{-c})\rho^{c}= (C\ve_0 K^c\del_0^{-c})\mathsf{d}(x,y)^c
$$
where $c=\frac{|\ln \del|}{\ln K}>0$.

 This shows  H\"older continuity of $h_j.$
The proof that $h_j^{-1}$ is H\"older continuous is similar except we use the equivalence relations induced by $W_j^u(x,\ve)$  and by $W_j^s(x, \ve).$
\end{proof}

\section{Perturbation theorem}
\label{AppPert}
\subsection{The statement}

Let $(B_j, \|\cdot\|_j)_{j\geq 0}$  be a sequence of Banach spaces and \\$A_j:B_j\to B_{j+1}$ be a sequence of bounded linear operators. We assume here that   there are $(\la_{j},h_{j},\nu_{j})\in(\bbC\setminus\{0\})\times B_j\times B_j^*$
where $B_j^*$ denotes the dual space of $B_j$
  so that 
\begin{equation}\label{u}
A_j h_{j}=\la_{j}h_{j+1},\quad (A_j)^* \nu_{j+1}=\la_{j}\nu_{j}.
\end{equation}
We will also assume that $\text{dim}(B_j)>1$ and that $\nu_j(h_j)=1$.

Let $\textbf{C}$ be a complex Banach space.
Let $\ve>0$. Denote by $B_{\textbf{C}}(0,\ve)$ the open ball  in $\textbf{C}$ around $0$ with radius $\ve$. 
 For each $t\in B_{\textbf{C}}(0,\ve)$ let $A_j^{(t)}:B_j\to B_{j+1}$ be a sequence of bounded operators
such that $A_j^{(0)}=A_j$.
 Set 
 $\DS A_j^{t,n}=A_{j+n-1}^{(t)}\circ\dots\circ A_j^{(t)}$. In the sequel, 
 for any sequence of operators $R_j:B_j\to B_{j+1}$ we will denote
 $\DS R_{j}^n=R_{j+n-1}\circ\cdots\circ R_{j+1}\circ R_j$. Write $\bbC'=\bbC\setminus\{0\}$.

 \begin{assumption}\label{Ass}
(i) $\DS \inf_{j}|\la_{j}|>0$.

\noindent
(ii) There are constants $ C_0,\delta_0\!\!>\!\!0$ such that  for all $j$ and  $g\in B_j$ so that $\nu_{j}(g)\!\!=\!\!0$  we have
\begin{equation}\label{ExCon z=0.0}
\|A_j^n g\|\leq C_0\|g\|\del_0^n.
\end{equation}
Moreover,  $\DS \liminf_{n\to\infty}\inf_{j}|\la_{j,n}|^{1/n}>\del_0$, where $\DS \la_{j,n}=\prod_{k=j}^{j+n-1}\la_k$.

\noindent
(iii)  The maps $t\to A_j^{(t)}$ are analytic in some neighborhood $\cU$ of the origin (which does not depend on $j$)   and $\DS \sup_{t\in\cU}\|A_j^{(t)}\|$ are uniformly bounded in $j.$
\end{assumption}


 The main result in this section is as follows. 
\begin{theorem}[A sequential perturbation theorem]\label{PertThm}
Under Assumption \ref{Ass}, for every $\del_1> \del_0$ 
there exist $r_1,C_1>0$ such that for every $j\in \bbN$ and $t\in B_{\textbf{B}}(0,r_1)$  we have the following.

(i) There are triplets  $(\la_{j}(t),h_{j}^{(t)},\nu_{j}^{(t)})\in\bbC'\times B_j\times B_j^*$,  which are uniformly bounded in $j,t$ so that 
\[
A_j^{(t)} h_{j}^{(t)}=\la_{j}(t)h_{j+1}^{(t,y)},\quad
(A_j^{(t)})^* \nu_{j+1}^{(t)}=\la_{j}(t)\nu_{k,j}^{(t,y)}.
\]
Moreover, $\la_{j}(0)=\la_{j}$, $h_{j}^{(0)}=h_{j}$, $\nu_{j}^{(0)}=\nu_{j}$,
and the triplets are analytic functions of  $t$ and their norm is uniformly bounded.
Furthermore, $\nu_{j}(h_{j}^{(t)})=\nu_{j}^{(t)}(h_{\ell,j})=\nu_j^{(t)}(h_j^{(t)})=1$.
If instead of analyticity of $t\mapsto A_j^{(t)}$ 
we assume that
 $\mathbf{C}$ is a real Banach space and 
$t\mapsto A_j^{(t)}$
 is $C^k$ for some $k\in\bbN$ with uniform bounds on the $C^k$ norms, then the above triplets are $C^k$ functions of $t$ with uniformly bounded $C^k$ norms.

(ii) 
Consider the operators $P_j^{(t)}:B_{j}\to B_{j+1}$ given by 
$$
P_j^{(t)} g=\la_{j}\nu_{j}^{(t)}(g)h^{(t)}_{j+1}.
$$
Then 
$$
P_j^{t,n}=\la_{j,n}(t)\nu_{j}^{(t)}(\cdot)h_{j+n}^{(t)}
$$
and 
denoting $E_j^{(t)}=A_j^{(t)}-P_j^{(t)}$ we have
$
A_j^{t,n}=P_j^{t,n}+E_j^{t,n}
$
and 
\begin{equation}\label{EX z}
\|E_j^{t,n}\|\leq C_1\del_1^n.
\end{equation}
\end{theorem}

\subsection{Preparations for the proof of Theorem \ref{PertThm}}

Our first result shows that  we can always consider two sided sequences instead of one sided ones.
\begin{lemma}
We can extend both sequences  $(B_j, \|\cdot\|_j)_{j\geq 0}$ and $(A_j)_{j\geq 0}$ to two sided sequences such that \eqref{u} and Assumption \ref{Ass} still hold.
\end{lemma}
\begin{proof}
For $j<0$ we take $B_j=B_0$. Then we take a nonzero  $\nu_{-1}\in B_0^*$ and set $\nu_{j}=\nu_{-1}$ for $j<-1$. We also define $\la_j=1+\del_0$ for $j<0$. Finally, we take a nonzero $h_{-1}\in B_0$ such that $\nu_{-1}(h_{-1})=1$ and define $h_j=h_{-1}$ for all $j<0$. For $j<0$, define $A_j(g)=(1+\del_0)\nu_{-1}(g)h_{j+1}=\la_j\nu_j(g) h_{j+1}$. It is clear that \eqref{u} holds with the new two sided sequences. Assumption \ref{Ass} is in force since for $j<0$, if $\nu_j(g)=0$ then $A_j(g)=0$ and so $A_j^ng=0$.
 \end{proof}
 \noindent
 Henceforth we assume that we have a two sided sequence satisfying  \eqref{u} and Assumption~\ref{Ass}.

\begin{remark}
If the operators $A_j$ are defined for all $j\in \bbZ$ then the proof of Theorem \ref{PertThm} 
(which proceeds by applying the Implicit Function Theorem) shows that the projections $P_j$ constructed in the theorem are unique.
The uniqueness does not hold if the operators are only defined for $j\in \bbN$ (cf.
the discussion after Theorem \ref{RPF 0}). 
\end{remark}

\begin{lemma}\label{OpLem}
(i) Let the operator $Q_j:B_j\to B_j$ be given by 
$$
Q_j(g)=\nu_{j}(g)h_{j}.
$$
Then for every $j$ and $g\in B_j$ we have $\nu_{j}(Q_j g)=\nu_{j}(g)$.

(ii) Consider the operator $P_j:B_j\to B_{j+1}$ given by 
$$
P_j g=\la_{j}\nu_{j}(g)h_{j+1}.
$$
Then $P_j$ coincides with $A_j$ on $\eta_j:=\text{span}\{h_{j}\}$ and
$$
A_{j+1}P_j=P_{j+1}A_j=P_{j+1}P_j,\text{ and }\ P_j^n=\la_{j,n}\nu_{j}(\cdot)h_{j+n}.
$$

Therefore, with $E_j=A_j-P_j$ we have
$$
E_j^n=A_j^n-P_j^n=A_j^n-\la_{j,n}\nu_{j}(\cdot)h_{j+n}.
$$
In particular,
$\DS
A_j^n=P_j^n+E_j^n.
$
\end{lemma}

\begin{proof}
(i)
$\DS
\nu_{j}\left(\nu_{j}(g)h_{j}\right)=\nu_j(g)\nu_j(h_j)=\nu_j(g).
$
\vskip0.1cm

 (ii) We have $P_j h_{j}=\la_j\nu_j(h_j)h_j=\la_j h_{j+1}$. So $A_j$ and $P_j$ coincide on $\eta_j$. 
Next, the identity $P_j^n=\la_{j,n}\nu_{j}(\cdot)h_{j+n}$ follows by induction on $n$. For $n=1$ is it just the definition of $P_j$, and to prove the induction step, if  $P_j^n=\la_{j,n}\nu_{j}(\cdot)h_{j+n}$   then 
$$
P_j^{n+1}(g)=P_{j+n}(P_j^{n}g)=\la_{j,n}\nu_{j}(g)P_{j+n}(h_{j+n})=\la_{j,n}\nu_j(g)\la_{j+n}h_{j+n+1}=\la_{j,n+1}\nu_j(g)h_{j+n+1}.
$$
 Since $(A_j)^*\nu_{j+1}=\la_j\nu_j$ we have $A_{j+1}P_j=P_{j+1}A_j=P_{j+1}P_j$, 
because all three expression above coincide with the operator
$$
g\to \la_j\la_{j+1}\nu_j(g)h_{j+2}=P_j^2(g).
$$
Using this the proof that $E_j^n=A_{j}^n-P_{j}^n$ proceeds by induction on $n$.
\end{proof}

\begin{lemma}\label{ExConvLemma}
Suppose that Assumption \ref{Ass} holds  and let $C$ be a constant such that
$\DS \sup_j\|Q_j\|\leq C$. Then for all $g\in B_j$,
\begin{equation}\label{ExCon z=0}
\|A_j^n g-\la_{j,n}\nu_{j}(g)h_{j+n}\|=\|A_j^n-P_j^n\|
\leq (C+1)C_0\|g\|\del_0^n.
\end{equation}
\end{lemma}
\begin{proof}
Set $R_j g=g-Q_j g$. By Lemma \ref{OpLem},\;  $\nu_{j}(R_j g)=0$. So by \eqref{ExCon z=0.0} applied to $R_jg$ 

\hskip2.5cm $\DS 
\|A_j^n g-\la_{j,n}\nu_{j}(g)h_{j+n}\|=
\|A_j^n(R_j g)\|\leq (C+1)C_0\|g\|\del_0^n.
$
\end{proof}

\begin{lemma}\label{SpecRadLemma}
Let $R_j:B_j\to B_{j+1}$ be a sequence of operators. Suppose that there are constants 
$\del_0,C_0>0$ such that
 for all $j,n$ we have
$$
\|R_j^{n}\|\leq C_0 \del_0^n.
$$
Then  for every $\del_1>\del_0$ there exist constants $\ve_0>0$ and
$C_1>0$ with the following property. For any other family $S_j:B_j\to B_{j+1}$ of linear operators satisfying
 $$
 \sup_{j}\|R_j-S_j\|<\ve_0
 $$
 we have 
$$
\|S^{n}_j\|\leq C_1 \del_1^n.
$$
\end{lemma}

\begin{proof}
Let $\del_1>\del_0$ and let $n_0$ be so that $C_0\del_0^{n_0}<\frac12\del_1^{n_0}$. Let $\ve_0>0$ be so small such that
$$
 \sup_{j}\|R_j-S_j\|<\ve_0
$$
implies $\DS \sup_j\|S_j^{n_0}\|<\del_1^{n_0}$. It is indeed possible to find such $\ve_0$ since 
$$
\|S_j^{n_0}-R_j^{n_0}\|
\leq C(C_0,\del_0,n_0)\ve_0
$$
for some constant $C(C_0,\del_0,n_0)$ which depends only on $C_0,\del$ and $n_0$.

Now, given $n\in\bbN$ let us write $n=kn_0+\ell$ for some $0\leq \ell<n_0$. Then\\
$\DS
\|S_j^{n}\|\leq \prod_{u=0}^{k-1}\|S_{j+un_0}^{n_0}\|\cdot \|S_{j+kn_0}^\ell\|\leq (\ve_0+C_0\del_0)^{\ell}\del_1^{-\ell}\del_1^n.
$
\end{proof}

\subsection{Proof of Theroem \ref{PertThm}}
 
\begin{proof}
We first prove the existence of triplets.  Denote
$$
 \bH_{j}=\{g\in B_j:\,\nu_{j}(g)=0\}
$$ 
Since $A_j^*\nu_{j+1}\!\!=\!\!\la_j\nu_{j}$ we have 
$\DS
A_j \bH_{j}\subset \bH_{j+1}.
$
Also, for  every $g_j\in \bH_j$ we have $\DS \nu_{j}(h_{j}+g_j)\!\!=\!\!1$
and 
$$
\nu_{j+1}\big(A_j(h_j+g_j)\big)=\la_{j}(A_j^*\nu_{j+1})(h_{j}+g_j)=
\la_{j}\nu_{j}(h_{j}+g_j)=\la_{j}.
$$
Consider the space $\textbf{H}$ of sequences $\bar g=(g_j)$ so that $g_j\in \bH_j$ for all $j$, 
equipped with the norm $\DS \|\bar g\|:=\sup_j \|g_j\|<\infty$.
Define a function $F:\textbf{C}\times\cY\times \textbf{H}\to\textbf{H}$ by
$$
\left(F(t,\bar g)\right)_j=A_{j-1}^{(t)}(h_{j-1}+g_{j-1})-\nu_{j}\big(A_{j-1}^{(t)}(h_{j-1}+g_{j-1})\big)\left(h_{j}+g_j\right).
$$
Then $F_k(0,0)=0$, and 
 $F$  is analytic in $t$ (in the case $C^k$ dependence on $t$ this map is differentiable $k$ times).  

We claim that  the partial derivative $\Phi:=(\Pt F/\Pt \textbf{H})(0, 0)$ 
is an isomorphism. Indeed, a direct computation shows that
$\DS
\left(\Phi \bar g\right)_j=A_{j-1}g_{j-1}-\la_{j}g_j.
$
Assume first that there is $\bar g\in\textbf{H}$ so that 
\begin{equation}\label{e1}
 A_{j-1}g_{j-1}=\la_{j}g_j.   
\end{equation}
When considered as maps $A_{j}:\textbf{H}_j\to\textbf{H}_{j+1}$ we have 
\begin{equation}\label{e2}
\|A_j^{n}g_j\|\leq C_0\|g_j\|\del_0^n    
\end{equation}
since on $\textbf{H}_j$ the functionals $\nu_{j}$ vanish. Iterating  \eqref{e1} we see that $A_j^{n}g_j=\la_{j,n}g_{j+n}$. Plugging in $j-n$ instead of $j$ in \eqref{e2}  we get that $g_j=\frac{A_{j-n}^ng_{j-n}}{\la_{j-n,n}}$ and so
$$
\|g_j\|\leq C_0\|\bar g\|\frac{\del_0^n}{|\la_{j,n}|}\leq C'\ve^n\to0
$$ 
for some $\ve\in(0,1)$ and a constant $C'$, 
where in the second inequality we have used Assumption \ref{Ass}(ii). 
 Hence $\bar g=0$.

Now we show that $\Phi$ is surjective. Take $\bar g\in \textbf{H}$. Set 
\[
h_j=-\sum_{n=0}^{\infty}\frac{A^n_{j-n}g_{j-n}}{\la_{j-n,n}}.
\]
This  series converges uniformly in $j$ since 
$$
\left\|\frac{A^{n}_{j-n}g_j}{\la_{j-n,n}}\right\|\leq C'\|\bar g\|\ve^{n}
$$
where $\ve$ and $C'$ are like in the previous estimates.
It is immediate that
$$
A_{j-1}h_{j-1}-\la_{j}h_j=\la_{j}g_j.
$$
Applying the implicit function theorem in Banach spaces (e.g. \cite[Theorem 3.2]{IFT})
 we get that there is $\ve_0>0$ so that for every $t\in B_{\textbf{C}}(0,\ve_0)$ there is $\bar g^{(t)}=(g_{j}^{(t)})_j\in \textbf{H}$ 
so that 
$$F_k(t,\bar g^{(t)})=0.$$ 	
Moreover, the function $t\to\bar g^{(t)}$ is analytic 
and so by possibly reducing $\ve_0$ we can assume that it is bounded 
(in the case of $C^k$ dependence on $t$ we get that this map is of class $C^k$).


Now we construct the families $\nu_{j}^{(t)}$. First, we obtain from \eqref{ExCon z=0} that
\begin{equation}\label{ExCon z=0, dual}
\|(A_j^n)^*-\la_{j,n}h_{j+n}^* \nu_{j}\|\leq C_0\del_0^n
\end{equation}
where  $h_{j+n}^*(\mu)=\mu(h_{j+n})$.
Let $\textbf{E}_j\subset B_j^*$ be the space of all functionals $\mu_j$ such that $\mu_j(h_{j})=0$. Then 
$\DS
(A_j)^*\textbf{E}_{j+1}\to \textbf{E}_j 
$
 since 
 $$
 (A_{j})^*\mu_{j+1}(h_{j})\!\!=\!\!\mu_{j+1}(A_j h_{j})=\la_{j}\mu_{j+1}(h_{j+1})=0.
$$

Let $\textbf{E}$ be the Banach space of sequences $\bar \mu=(\mu_j)$ such that $\mu_j\in \textbf{E}_j$ for all $j$ and $\DS \|\bar\mu\|:=\sup_j\|\mu_j\|<\infty$. Define 
$$
\big(G(t,\bar\mu)\big)_j=(A_j^{(t)})^*(\nu_{j+1}+\mu_{j+1})-\la_{j}(t)(\nu_{j}+\mu_{j})
$$
where $\lambda_j(t)=\nu_j(h_j^t).$
Then $G(0,0)=0$ and $G$ is a well defined  function on $\{t: \|t\|<\ve_0\}$ for some $\ve_0>0$. Moreover, $G$ is  analytic in $t$.
We consider $G$ as a function from a neighborhood of $(0,0)\in\textbf{C} \times\textbf{E}$ to $\textbf{E}$.

Now, a direct computation shows that the derivative of $G$ at $(0,0)$ in the direction $\textbf{E}$ is 
the operator given by
$$
\Psi:=(\Pt G_k/d\textbf{E})(0,0)\bar\mu)_j=(A_j)^*\mu_{j+1}-\la_{j}\mu_{j}.
$$
Let us show that $\Psi$ is injective. If $\bar\mu$ belongs to its kernel 
 then $(A_j)^*\mu_{j+1}=\la_{j}\mu_j$, and so $(A_j^n)^{*}\mu_{j+n}=\la_{j,n}\mu_{j}$.  
 Using \eqref{ExCon z=0, dual} we get that 
$\DS
\|\mu_j\|\leq C\ve^n
$
for some constants $\ve\in(0,1)$ and $C>0$. Therefore $\bar\mu=0$. 

Next, $\Psi$ is also surjective since for any $\bar\mu\in \textbf{E}$ we have
$\DS
(A_j)^*\ka_{j+1}-\la_{j}\ka_{j}=\la_{j}\mu_j$
where
$\DS \ka_{j}=-\sum_{n=0}^\infty \frac{(A_{j+1}^n)^*\mu_{j+n+1}}{\la_{j,n}}
$  (the convergence of this series follows from \eqref{ExCon z=0, dual}).

Hence, by the implicit function theorem there is an analytic
 function  $t\mapsto \bar\ka^{(t)}=(\ka_{j}^{(t)})_{j}$ with values in $\textbf{E}$ so that 
$$
(A_j^{(t)})^*(\nu_{kj+1}+\ka_{j+1}^{(t)})-\la_{j}(t)(\nu_{j}+\ka_{j}^{(t)})=0.
$$
Take $\nu_{j}^{(t)}=\nu_{j}+\ka_{j}^{(t)}$
and notice that $\nu_{j}^{(t)}(h_{j}^{(t)})$ does not depend on $j$. Indeed,
we have 
$$
\la_j(t)\nu_{j+1}^{(t)}(h_{j+1}^{(t)})=\nu_{j+1}^{(t)}(\la_j(t)h_{j+1}^{(t)})
=\nu_{j+1}^{(t)}(A_j^{(t)}h_j^{(t)})=(A_j^{(t)})^*\nu_{j+1}^{(t)}(h_j^{(t)})=\la_j(t)\nu_{j}^{(t)}(h_j^{(t)}).
$$
We conclude that there is a  function $c(t)$ such that $\nu_{j}^{(t)}(h_{j}^{(t)})=c(t)$ for every $j\geq0$. The functions $c(t)$ are bounded and analytic and $ c(0)=1$. Thus, by decreasing $\|t\|$ if necessary we can assume that  $|c(t)|\geq\frac12$. Next, by replacing $h_{j}^{(t)}$ with $h_{j}^{(t)}/c(t)$ (or $\nu_{j}^{(t)}$ by $\nu_{j}^{(t)}/c(t)$) we can just assume that $c(t)=1$.  Like in the construction of $g^{(t)}$, in the case $t\to (A_j^{(t)})$ is of class $C^k$ all the above functions of $t$ are differentiable $k$ times.
\vskip0.2cm

Finally, the exponential convergence follows by Lemma \ref{SpecRadLemma} applied with $R_j=A_j-P_j$ and $S_j=A_j^{(t)}-P_{j}^{(t)}$ (when $t$ is close enough to $0$)
where 
$\DS
P_{j}^{(t)}=\la_{j}(t)\nu_{j}^{(t)}(\cdot)h_{j+1}^{(t)}
$
taking into account that by Lemma \ref{OpLem},
$\DS
S_j^n=A_j^{t,n}-P_{j}^{t,n}=A_{j}^{t,n}-\la_{j,n}(t)\nu_{j}^{(t)}(\cdot)h_{j+n}^{(t)}.
$
\end{proof}

\end{document}